\documentclass[final]{siamart220329}

\usepackage{lipsum}
\usepackage{graphicx}
\usepackage{subcaption}
\usepackage{epstopdf}
\usepackage{algorithmic}
\usepackage{amsfonts}              % for blackboard bold, etc
\usepackage{amsmath}
\usepackage{amssymb}
\usepackage{blkarray}
\usepackage{enumitem}
\usepackage[normalem]{ulem}

\ifpdf
  \DeclareGraphicsExtensions{.eps,.pdf,.png,.jpg}
\else
  \DeclareGraphicsExtensions{.eps}
\fi

\newcommand{\TheTitle}{On the Two-parameter Matrix pencil Problem} 
\newcommand{\TheAuthors}{S. K. Gungah, F. F. Alsubaie, and I. M. Jaimoukha}

\headers{\TheTitle}{\TheAuthors}

\title{{\TheTitle}\thanks{Final version submitted to the editors April 9, 2024.}}

\author{
  Satin K. Gungah\thanks{Control and Power Group, Department of Electrical and Electronic Engineering, Imperial College, London SW7 2AZ, UK (\email{s.gungah@imperial.ac.uk}).}
  \and
  Fawwaz F. Alsubaie\thanks{The author carried out the work while he had been affiliated with the Control and Power Group, Department of Electrical and Electronic Engineering, Imperial College London SW7 2BT, UK (\email{f.alsubaie15@alumni.imperial.ac.uk}).}
  \and
  Imad M. Jaimoukha \thanks{Control and Power Group, Department of Electrical and Electronic Engineering, Imperial College, London SW7 2AZ, UK (\email{i.jaimouka@imperial.ac.uk}).}
}

\newcommand{\be}{\begin{equation}}
\newcommand{\ee}{\end{equation}}
\newcommand{\bea}{\begin{eqnarray}}
\newcommand{\eea}{\end{eqnarray}}
\newcommand{\bean}{\begin{eqnarray*}}
\newcommand{\eean}{\end{eqnarray*}}

\newtheorem{remark}{Remark}[section]

\def\l{{\lambda}}
\def\a{{\alpha}}
\newcommand{\K}{\mathcal{K}}
\newcommand{\A}{{A}}

\newcommand{\D}{{D}}
\newcommand{\E}{{A}}
\newcommand{\F}{{F}}
\newcommand{\I}{{I}}

\newcommand{\T}{{T}}
\newcommand{\U}{\mathcal{U}}
\newcommand{\V}{\mathcal{V}}

\newcommand{\Y}{\mathcal{Y}}

\renewcommand{\F}{\mathcal{F}}
\renewcommand{\S}{\mathcal{S}}
\renewcommand{\T}{\mathcal{T}}
\newcommand{\cH}{\mathcal{H}}
\newcommand{\tm}{\tilde{m}}
\newcommand{\tn}{\tilde{n}}

\renewcommand{\D}{\Delta}
\newcommand{\Dt}{\tilde{\Delta}}
\newcommand{\G}{\Gamma}

\renewcommand{\o}{{0}}

\renewcommand{\u}{{u}}

\newcommand{\x}{{x}}
\newcommand{\y}{{y}}

\renewcommand{\vec}{{\rm vec}}

\newtheorem{thm}{Theorem}[section]

\newtheorem{prob}[thm]{Problem}

\DeclareMathOperator{\vect}{vec}

\DeclareMathOperator{\rrank}{rank}

\newcommand{\RR}{\mathbb{R}}      % for Real numbers
\newcommand{\CC}{\mathbb{C}}      % for Complex
\newcommand{\textmultiset}[2]{\bigl(\!{\binom{#1}{#2}}\!\bigr)}
\newcommand{\displaymultiset}[2]{\left(\!{\binom{#1}{#2}}\!\right)}
\newcommand\multiset[2]{\mathchoice{\displaymultiset{#1}{#2}}
                                {\textmultiset{#1}{#2}}
                                {\textmultiset{#1}{#2}}
                                {\textmultiset{#1}{#2}}}

\newcommand{\matOneTwo}[2]{\left[\begin{array}{cc} #1 & #2 \end{array} \right]}
\newcommand{\matTwoOne}[2]{\left[\begin{array}{c} #1 \\ #2 \end{array} \right]}
\newcommand{\matTwoTwo}[4]{\left[\begin{array}{cc} #1 & #2 \\ #3 & #4 \end{array} \right]}

\newcommand{\matOneThr}[3]{\left[\begin{array}{ccc} #1 & #2 & #3 \end{array} \right]}
\newcommand{\matThrOne}[3]{\left[\begin{array}{c} #1 \\ #2 \\ #3 \end{array} \right]}

\newcommand{\matFourOne}[4]{\left[
\begin{array}{c}
#1 \\
#2 \\
#3 \\
#4
\end{array}
\right]}

\newcommand{\matThrThr}[9]{\left[
\begin{array}{ccc}
#1 & #2 & #3 \\
#4 & #5 & #6 \\
#7 & #8 & #9
\end{array}
\right]}

\ifpdf
\hypersetup{
  pdftitle={\TheTitle},
  pdfauthor={\TheAuthors}
}
\fi

\begin{document}
\maketitle

% REQUIRED
\begin{abstract}
The multiparameter matrix pencil problem (MPP) is a generalization of the one-parameter MPP: given a set of  $m\times n$ complex matrices $A_0,\ldots, A_r$,  with $m\ge n+r-1$, it is required to find all complex scalars $\l_0,\ldots,\l_r$, not all zero,  such that the matrix pencil $A(\l)=\sum_{i=0}^r\l_iA_i$ loses column rank and the corresponding nonzero complex vector $x$ such that $A(\lambda)x=0$.
% We call the $({r}+1)$-tuple $\l=(\l_0,\ldots,\l_r)$ an eigenvalue and the corresponding vector $x$ an eigenvector.
This problem is related to the well-known multiparameter eigenvalue problem except that there is only one pencil and, crucially, the matrices are not necessarily square.
% This paper uses our preliminary investigation in 
%  {\it F. F. Alsubaie, $\mathcal{H}_2$ Optimal Model Reduction for Linear Dynamic Systems and the Solution of Multiparameter Matrix Pencil Problems, PhD thesis, Imperial College London, 2019}, which presents a theoretical study of the multiparameter MPP and its applications in the $\mathcal{H}_2$ optimal model reduction problem, to give a full solution to the two-parameter MPP.
 In this paper, we give a full solution to the two-parameter MPP.
 Firstly, an inflation process is implemented to show that the two-parameter MPP is equivalent to a set of three $m^2\times n^2$ simultaneous one-parameter MPPs. These problems are given in terms of Kronecker commutator operators (involving the original matrices) which exhibit several symmetries. These symmetries are analysed and are then used to deflate the dimensions of the one-parameter MPPs to $\frac{m(m-1)}{2}\times\frac{n(n+1)}{2}$ thus simplifying their numerical solution. In the case that $m=n+1$ it is shown that the two-parameter MPP has at least one solution and generically $\frac{n(n+1)}{2}$ solutions and furthermore that, under a rank assumption, the Kronecker determinant operators satisfy a commutativity property. This is then used to show that the two-parameter MPP is equivalent to a set of three simultaneous eigenvalue problems.
 % of dimension $\frac{n(n+1)}{2}\times\frac{n(n+1)}{2}$.
 A general solution algorithm is presented and numerical examples are given to outline the procedure of the proposed algorithm.   
\end{abstract}

% REQUIRED

% REQUIRED

\section{Notation}\label{sec:notation}
Upper case letters denote matrices while lower case letters denote vectors or scalars.  $\RR$ and $\CC$ denote the set of real and complex numbers, $\RR^n$ and $\CC^n$ the sets of real and complex $n$-dimensional column vectors and $\RR^{m\times n}$ and $\CC^{m\times n}$ the sets of real and complex $m\times n$ matrices, respectively. For $A\in\CC^{m\times n}$, $A^{T}$ and $\rrank(A)$ denote transpose and rank, respectively. ${\I_{r}}$ and $0_{r\times s}$ denote the $r\times r$ identity matrix and the $r\times s$ zero matrix, respectively, with the subscripts dropped when they can be inferred from the context. The vector operator and Kronecker product are denoted as $\vect{(\cdot)}$ and $\otimes$, respectively, where we take $\vec{(A)}$ to be the vector of the stacked columns of $A$. For $x\!\in\!\CC^n$ we define $x^{\otimes2}\!=\!x\otimes x$ and for all $i\!\ge\!2$, $x^{\otimes(i+1)}\!=\!x\otimes x^{\otimes i}$. A vector $z\!\in\!\CC^{n^r}$ is called strongly decomposable if $z\!=\!x^{\otimes r}$ for some $x\!\in\!\CC^{n}$.

\section{Introduction}
The multiparameter matrix pencil problem (MPP) of the form
\be\label{eqn:Pencil}
\left(\sum_{i=0}^r\l_i A_i\right)x=0,\quad 0\ne\l=\left[\begin{array}{c}\l_0\\\vdots\\\l_r\end{array}\right]\in\CC^{r+1},\quad 0\ne x\in\CC^n.
\ee
where $A_{i}\in\CC^{m\times n}$, $i=0,\ldots,r$, are given (not necessarily square) matrices has received increasing attention in recent years. In \cite{blum1978c,blum1978}, a technique based on gradient method is presented to solve a given multiparameter MPP, where local convergence is proved under certain conditions. Numerical examples are given to illustrate the proposed algorithm.
Another gradient based algorithm was proposed in \cite{wicks1995}, where the multiparameter MPP is posed as a structured matrix perturbations that cause a specified system matrix to fail to have full column rank. In \cite{khazanov1997}, the notion of the generalized Jordan form and generating vector for the multiparameter MPP were considered. Various numerical methods based on iterative algorithms have been discussed in \cite{Khazanov2005}. In \cite{Khazanov2007}, a linearization process is performed on a given multiparameter polynomial matrix to yield a multiparameter MPP with linear eigenvalues. The spectral properties between the multiparameter polynomial MPP and the linear MPP is studied. An iterative algorithm with minimal residual  quotient and Newton shift is outlined to the MPP in-hand. The problem was also considered in \cite{shapiro2009} where under certain assumptions on the pencil in \eqref{eqn:Pencil}, they derive results concerning the number of solutions in the generic case. In \cite{Ahmad2010} the first order $\mathcal{H}_2$ optimal model reduction problem was shown to be equivalent to a one-parameter MPP. In \cite{Ahmad2011} the second order $\mathcal{H}_2$ optimal model reduction problem was shown to be equivalent to a two-parameter MPP, but no dedicated algorithms were given for its solution. In \cite{Alsubaie2019} it was shown that the $r$-th order $\mathcal{H}_2$ optimal model reduction problem is equivalent to an $r$-parameter MPP. Furthermore, a general approach was outlined for the numerical solution of the $r$-parameter MPP, subject to some regularity assumptions. In \cite{demoor2019} it is shown that the low order least squares optimal realization of linear time-invariant dynamical systems from
given data can be posed as a multiparameter MPP. Similarly, it is shown in \cite{vermeersch2019} that globally optimal least-squares
identification of autoregressive moving-average
models is equivalent to a multiparameter MPP. Both \cite{demoor2019} and \cite{vermeersch2019} then propose solving this problem via the block Macaulay method. In \cite{decock2021} several eigenvalue problems, including the multiparameter MPP, were considered. They use shift-invariant subspaces and multi-dimensional realization algorithms to provide a unified framework for solving these problems. In \cite{vermeersch2022}, two algorithms, based on the block Macaulay matrix approach \cite{demoor2019,decock2021}, were proposed to transform the multiparameter MPP into standard eigenvalue problems. The work in \cite{hochstenbach2024} uses a randomized sketching approach and exploits the deflation technique in \cite{Alsubaie2019} to transform the multiparameter MPP to square multiparameter eigenvalue problems, see \eqref{eqn:Aij} below. They then use available algorithms for the multiparameter eigenvalue problem to solve these transformed MPPs.

The related multiparameter eigenvalue (MEV) problems of the form
\begin{equation}\label{eqn:Aij}
\left(\sum_{i=0}^{r} \l_i {\E}_{ij}\right) {\x_j}={\o},~\l_0,\l_1,\ldots,\l_r\in\CC,\qquad 0\ne x_j\in\CC^n,~j=1,\ldots,r,
\end{equation}
where the $A_{ij}\in\CC^{n\times n}$ are given square matrices have received considerable attention since the last century and during more recent years \cite{Atkinson1968,Atkinson1972,gadzhiev1987,volkmer2006,atkinson2010,Hoch2018}.
% They arise in many scientific areas from chemistry and mathematical physics to computing and economics.
% \cite{pons2017,pons2015,molzahn2013,molzahn2010,sakaue2016,beilina2018,Atkinson1968,Atkinson1972,gadzhiev1987,faierman1979,volkmer2006,atkinson2010,ugurlu2018,Muhic2010,Muhic2009,muhic2008,Hoch2018}. 
Although the multiparameter MPP we consider is different from the MEV problem, notably in that the matrices are tall and there is only one pencil, certain aspects of these problems are similar and we will use some of the machinery developed in \cite{Atkinson1972} for the solution of the MEV problem, suitably extended, to the non-square multi-parameter MPP. In \cite{Carmichael1921a,Carmichael1921b,Carmichael1922,Atkinson1968,Atkinson1972}, multiparameter spectral theory was firstly introduced. This involves the solution of a system of multiparameter linear eigenvalue problems of $r$-tuples. It is shown that the root vectors \cite{Atkinson1972} associated with a regular matrix pencil is decomposable in terms of its right eigenvectors of the underlying multiparameter system. Furthermore, the determinant operators associated with the regular pencil commute. In \cite{binding1989,sleeman1978,kosir1994,kosir2004,browne1972},  the solvability, regularity, and further classification of the multiparameter eigenvalue problem are discussed. 

In this paper, we use our preliminary investigation in \cite{Alsubaie2019}, which presents a theoretical study of the problem and its applications in the $\mathcal{H}_2$ optimal model reduction problem \cite{Serkan2008,beattie2017,Ahmad2010,Ahmad2011}, to give a full solution to the two-parameter MPP. Section~\ref{sec:3} gives a statement of the problem. It then provides the solution of the problem with one less matrix: for $A,B\in\CC^{m\times n}$, the matrix pencil $\nu_1A+\nu_2B$ loses rank if and only if the Kronecker commutator $\D=A\otimes B-B\otimes A$ loses rank and that the null space of $\D$, if nonempty, always contains a strongly decomposable vector. This result is then used in an inflation process to prove that the
two-parameter MPP is equivalent to a set of three $m^2\times n^2$ simultaneous one-parameter MPPs. These problems are given in terms of Kronecker commutator operators in the form of $\D$ defined above. These operators exhibit several symmetries and these are analysed in detail in Section~\ref{sec:Commutator}. Starting with the commutation matrix introduced in \cite{Tracy1969} and the symmetric projection and elimination matrices introduced in \cite{Magnus1979,Magnus1980}, we define a skew-symmetric projection matrix and three selection matrices. These are used to define two real orthogonal sparse matrices. It is shown that all Kronecker commutator operators of a given dimension can be block anti-diagonalized using these two matrices. Furthermore, $\D$ loses rank if and only if one of the anti-diagonal blocks loses rank. This result is then used in Section~\ref{sec:Deflation} to deflate the dimensions of the one-parameter MPPs to $\frac{m(m-1)}{2}\times\frac{n(n+1)}{2}$, thus simplifying their numerical solution. Section~\ref{sec:m=n+1} deals with the case $m=n+1$ showing that a solution to the two-parameter MPP always exists, and that it generically has $\frac{n(n+1)}{2}$ solutions. It also establishes, under a rank assumption, a commutativity property of the underlying deflated matrices, thus establishing that the two-parameter MPP is equivalent to a set of three $\frac{n(n+1)}{2}\times\frac{n(n+1)}{2}$ simultaneous eigenvalue problems. Section~\ref{sec:Examples} then summarises the contribution of the paper by presenting a solution algorithm and a few numerical examples to highlight the procedure of the proposed algorithm. Future research directions towards extending our solution approach to the general multiparameter MPP  are outlined in Section~\ref{sec:Future}. Finally our conclusions appear in Section~\ref{sec:Conclusions}.

\section{The Two-Parameter MPP}\label{sec:3}
In this section, we define and give a general solution to the two-parameter MPP:
\begin{prob}\label{prob:m=2}
Let $A_0,A_1,A_2\in\CC^{{m}\times{n}}$ be given where ${m}>{n}$ and assume that
\be\label{eqn:Assumptions}
\rrank\left(\left[\begin{array}{c}A_0\\A_1\\A_2\end{array}\right]\right)={n},\qquad\rrank\left(\left[\begin{array}{ccc}A_0&A_1&A_2\end{array}\right]\right)={m}.
\ee
Find all eigenvalues $\l$ and the corresponding eigenvectors $x$ such that 
\be\label{eqn:2-Parameter}
(\l_0A_0+\l_1A_1+\l_2A_2)x=0,\qquad 0\ne x\in\CC^{n},\qquad 0\ne\l:=\!\left[\!\begin{array}{c}\l_0\\\l_1\\\l_2\end{array}\!\right]\!\in\CC^3,
\ee
where we do not distinguish between an eigenvalue $\l$ and $\a\l$ for $0\ne\a\in\CC$ and between an eigenvector $x$ and $\beta x$ for $0\ne\beta\in\CC$.
\end{prob}
\begin{remark}\label{rem:Assumptions}
There is no loss of generality in making assumptions \eqref{eqn:Assumptions}. If they are not satisfied,  we can, using unitary transformations, apply a column compression \cite{Van1979} on the first (block column) matrix in \eqref{eqn:Assumptions} and/or a row compression on the second (block row) matrix in \eqref{eqn:Assumptions}, remove the zero columns and/or rows and get an equivalent pencil with smaller ${n}$ and/or $m$. While our results are valid without these assumptions, it is recommended to carry out the above procedure since the computational complexity of our solution increases rapidly with $m$ and $n$. Note also that the number of equations in \eqref{eqn:2-Parameter} is $m$ and the number of unknowns is $n\!+\!1$ ($n\!-\!1$ for nontrivial $x$ and 2 for nontrivial $\lambda$), hence our assumption that $m\!>\!n$: there are at least as many equations as unknowns. 
\end{remark}

Before we give a solution to Problem~\ref{prob:m=2}, we investigate the one-parameter MPP. We define the Kronecker commutator operator and investigate its properties. Further properties of the commutator operator will be given in Section~\ref{sec:Commutator} below.

\begin{thm}\label{thm:m=2-square}
Let $A,B\in\CC^{{m}\times{n}}$ and define the Kronecker commutator operator $\D=A\otimes B-B\otimes A\in\CC^{{m}^2\times {n}^2}$. Then the following two statements are equivalent
\be\label{eqn:(a)}
\D z=0,\qquad 0\ne z\in\CC^{{n}^2}\!.
\ee
\be\label{eqn:(b)}
(\nu_1A+\nu_2B)x=0,\qquad 0\ne x\in\CC^{n},\qquad 0\ne\!\left[\!\!\!\begin{array}{c}\nu_1\\\nu_2\end{array}\!\!\!\right]\!\in\CC^2\!.
\ee
Furthermore, if the null space $\mathcal{N}$ of $\Delta$ is nonempty, it includes a nonzero strongly decomposable vector. Finally, if $\mathcal{N}$ has dimension one, then $$\mathcal{N}=\{\alpha z:\alpha\in\CC,~z{\rm~is ~strongly~decomposable}\}.$$
\end{thm}

\begin{proof}
The proof of \eqref{eqn:(b)}$\Rightarrow$\eqref{eqn:(a)} was derived in Theorem~4.1 in \cite{Alsubaie2019}, but we include it here for completeness. If $Ax\!=\!0$ or $Bx\!=\!0$, then \eqref{eqn:(a)} follows with $z\!=\!x\otimes x$, so we can assume $Ax\!\ne\!0$ and $Bx\!\ne\!0$. By carrying out the manipulations
$$\bigl((\nu_1A+\nu_2B)x\bigl)\otimes(Bx)-(Bx)\otimes\bigl((\nu_1A+\nu_2B)x\bigl)=0,$$
and 
$$\bigl((\nu_1A+\nu_2B)x\bigl)\otimes(Ax)-(Ax)\otimes\bigl((\nu_1A+\nu_2B)x\bigl)=0,$$
we get
$\nu_1\D x^{\otimes2}=0$ and $\nu_2\D x^{\otimes2}=0$. Since $\nu_1$ and $\nu_2$ are not both zero, it follows that $\D x^{\otimes2}=0$ and this proves $\eqref{eqn:(a)}$. Next, we prove \eqref{eqn:(a)}$\Rightarrow$\eqref{eqn:(b)} by constructing $\nu_1,~\nu_2$ and $x\in\CC^{n}$ that satisfy \eqref{eqn:(b)}. Taking the inverse $\vect(\cdot)$ operation in \eqref{eqn:(a)}:
$$
BZA^T-AZB^T=0,\qquad\vec{(Z)}=z,\qquad 0\ne Z\in\CC^{{n}\times {n}}.
$$
Let $Z$ have rank $r$ where $1\!\le \!r\!\le\! {n}$. Then $Z\!=\!UV^T$ for some $U,V\!\in\!\CC^{{n}\times r}$ with full column rank. It follows that
$BUV^T\!A^T\!\!-\!AUV^T\!B^T\!\!=\!0$,
which can be written as
\be\label{eqn:AUAV2}
\left[\begin{array}{cc}BU&-AU\end{array}\right]\left[\begin{array}{cc}AV&BV\end{array}\right]^T=0.
\ee
Carrying out a row compression \cite{Van1979} on $AV$ and applying it to $BV$ we get:
\be\label{eqn:RowCompression1}
W \left[\begin{array}{c|c}AV&BV\end{array}\right]=\!\!\!\left.\begin{array}{rl}\left.\begin{array}{c}{\scriptstyle r_1}\\{\scriptstyle {m}-r_1}\end{array}\right.&\!\!\!\!\!\!\!\!\!\!\left[\begin{array}{c|c}\tilde{S}_{11}&\tilde{S}_{12}\\0&\tilde{S}_{22}\end{array}\right]\end{array}\right.\!\!\!\!\!,
\ee
where $r_1=\rrank(\tilde{S}_{11})=\rrank(AV)\le r$ and where $W\in\CC^{{m}\times {m}}$ is nonsingular. Next, we carry out a second row compression on $\tilde{S}_{22}$ and apply it to \eqref{eqn:RowCompression1}:
\be\label{eqn:RowCompression2}
\left[\begin{array}{cc}I_{r_1}&0\\0&\hat{W}\end{array}\right] W \left[\begin{array}{c|c}AV&BV\end{array}\right]=\left.\begin{array}{rl}\left.\begin{array}{c}{\scriptstyle r_1}\\{\scriptstyle r_2}\\{\scriptstyle {m}-r_1-r_2}\end{array}\right.&\!\!\!\!\!\!\!\!\!\!\left[\begin{array}{c|c}\tilde{S}_{11}&\tilde{S}_{12}\\0&\hat{S}_{22}\\0&0\end{array}\right]\end{array}\right.\!\!\!\!\!\!,
\ee
where $r_2=\rrank(\hat{S}_{22})=\rrank(\tilde{S}_{22})\le r$ and where $\hat{W}\in\CC^{({m}-r_1)\times({m}-r_1)}$ is nonsingular. There are two cases to consider:
\begin{enumerate}
\item If $r_1+r_2\le r$, then \eqref{eqn:RowCompression2} shows that the pencil $\nu_1AV\!+\!\nu_2BV$ is equivalent to an $(r_1\!+\!r_2)\times r$ pencil, with $r_1\!+\!r_2\!\le \!r$, and so has at least one solution (and might have up to $r$ finite solutions or a continuum of solutions). Let $(\nu_1AV\!+\!\nu_2BV)y\!=\!0$ be any solution. Then setting $x\!=\!Vy\!\ne\!0$ proves \eqref{eqn:(b)}. 
\item If $r_1+r_2>r$, we carry out a partitioning of \eqref{eqn:RowCompression2} as follows:
\be\label{eqn:RowCompression3}
\left[\begin{array}{cc}I_{r_1}&0\\0&\hat{W}\end{array}\right] W \left[\begin{array}{c|c}AV&BV\end{array}\right]=\!\!\!\!\left.\begin{array}{rl}\left.\begin{array}{c}{\scriptstyle r}\\{\scriptstyle r_1+r_2-r}\\{\scriptstyle {m}-r_1-r_2}\end{array}\right.&\!\!\!\!\!\!\!\!\!\!\left[\begin{array}{c|c}S_{11}&S_{12}\\0&S_{22}\\0&0\end{array}\right]\end{array}\right.\!\!\!\!\!\!,
\ee
so that $S_{22}$ has full row rank. It follows from \eqref{eqn:AUAV2} and \eqref{eqn:RowCompression3} that
\be\label{eqn:RowCompression4}
\left[\begin{array}{cc}BU&-AU\end{array}\right]\left[\begin{array}{cc}S_{11}^T&0\\S_{12}^T&S_{22}^T\end{array}\right]=0,
\ee
since $W$ and $\hat{W}$ are nonsingular. Equation \eqref{eqn:RowCompression4} allows us to derive two sets of solutions to \eqref{eqn:(b)}:
\begin{enumerate}
\item The equation corresponding to the second column of \eqref{eqn:RowCompression4} gives $AUS^T_{22}=0$. Since $S^T_{22}$ has full column rank, there exists $y\ne0$ such that $S^T_{22}y\ne0$. Then, $\nu_2=0$, $0\ne\nu_1\in\CC$ arbitrary and $x=US^T_{22}y\ne0$ satisfy \eqref{eqn:(b)}.
\item The equation corresponding to the first column of \eqref{eqn:RowCompression4} gives
\be\label{eqn:RowCompression5}
BUS^T_{11}-AUS^T_{12}=0,
\ee
where $S_{11},S_{12}\!\in\!\CC^{r\times r}$. Hence there exist $\nu_1,\nu_2\!\in\!\CC$, not both zero, and $0\!\ne\! y\!\in\!\CC^r$, such that $(\nu_1S^T_{11}\!+\!\nu_2S_{12}^T)y\!=\!0$.
Note that, by the sequential row compression construction above, $S_{11}^T$ and $S_{12}^T$ have no common null space, so that $S_{11}^Ty\!=\!0\!\Rightarrow \!S_{12}^Ty\!\ne\!0$. The remainder of the proof depends on whether $S_{11}^Ty\!=\!0$:
\begin{enumerate}
\item If $S_{11}^Ty\ne0$ so that $\nu_2\ne0$, post-multiplying \eqref{eqn:RowCompression5} by $\nu_2y$ gives
$$
\nu_2BUS^T_{11}y-\nu_2AUS^T_{12}y=0.
$$
Substituting $-\nu_2S_{12}^Ty=\nu_1S_{11}^Ty$ gives
$
(\nu_1A+\nu_2B)US_{11}^Ty=0,
$
and defining $x=US_{11}^Ty\ne0$ then gives \eqref{eqn:(b)}. 
\item If $S_{11}^Ty=0$, then $S_{12}^Ty\ne0,~\nu_2=0$ and $\nu_1\ne0$. Post-multiplying \eqref{eqn:RowCompression5} by $\nu_1y$ gives
$
\nu_1AUS^T_{12}y=0.
$
Defining $x=US_{12}^Ty\ne0$ then gives \eqref{eqn:(b)}. Note that in this case, $A$ loses column rank.
\end{enumerate}
\end{enumerate}
\end{enumerate}
This completes the proof that \eqref{eqn:(a)}$\Rightarrow$\eqref{eqn:(b)}. 

Next, we prove that if $\mathcal{N}$ is nonempty, it includes a nonzero strongly decomposable vector. Suppose $\mathcal{N}$ is nonempty so that $\D z=0$ for some $0\ne z\in\CC^{n^2}$. Then \eqref{eqn:(b)} is satisfied and the proof of \eqref{eqn:(b)}$\Rightarrow$\eqref{eqn:(a)} shows that $\D x^{\otimes2}=0$, $x\ne0$ and so $x^{\otimes2}\in\mathcal{N}$. This also proves the last part.
\end{proof}
\begin{remark}
Theorem~\ref{thm:m=2-square} extends Theorems~6.3.1 and 6.3.2 in \cite{Atkinson1972} to non-square matrices. In the case that the one-parameter MPP in \eqref{eqn:(b)} has no solution, \cite{boutry2005} defines the notion of the nearest solvable one-parameter MPP in terms of the size of matrix perturbations to $A$ and $B$. A possible future research direction would be to use our approach to extend this result to the multiparameter MPP.
\end{remark}

Next, we use Theorem~\ref{thm:m=2-square} to give a solution to Problem~\ref{prob:m=2}.
\begin{thm}\label{thm:1-Parameter}
Let $A_0,A_1,A_2\in\CC^{{m}\times{n}}$ be as given in Problem~\ref{prob:m=2} and define the Kronecker commutator operators
\be\label{eqn:Dis}
\Delta_0\!:=\!A_1\otimes A_2\!-\!A_2\otimes A_1,\quad
\Delta_1\!:=\!A_2\otimes A_0\!-\!A_0\otimes A_2,\quad
\Delta_2\!:=\!A_0\otimes A_1\!-\!A_1\otimes A_0,
\ee
so that $\D_0,\D_1,\D_2\in\CC^{m^2\times n^2}$. Then the ${m}\times {n}$ two-parameter MPP in \eqref{eqn:2-Parameter} has a solution if and only if the following three ${m}^2\times {n}^2$ one-parameter MPPs
\be\label{eqn:1-Parameter}
(\l_i\Delta_j\!-\!\l_j\Delta_i)z\!=\!0,~ (i,j)\!\in\!\{(0,1),(0,2),(1,2)\},\quad
 0\!\ne\!\!\left[\!\!\begin{array}{c}\l_0\\\l_1\\\l_2\end{array}\!\!\right]\!\!\!\in\!\CC^3,\quad
 0\!\ne\! z\!\in\!\CC^{{n}^2}\!,
\ee
have a simultaneous solution. Hence, if the one-parameter MPPs in \eqref{eqn:1-Parameter} have a simultaneous solution then there exists a strongly decomposable vector $0\ne z= x^{\otimes2}$ such that $(\l_i\Delta_j-\l_j\Delta_i) x^{\otimes2}=0,~(i,j)\in\{(0,1),(0,2),(1,2)\}$, and $(\l_0,\l_1,\l_2)\!\ne\!0$ such that $(\l_0A_0+\l_1A_1+\l_2A_2)x=0$. 
\end{thm}
\begin{proof}
First, we prove that \eqref{eqn:2-Parameter}$\Rightarrow$\eqref{eqn:1-Parameter}. If $A_0x=0$, then by taking $z= x^{\otimes2}$, $\mu_1=\mu_2=0$ and any $0\ne\mu_0\in\CC$, it can be verified that \eqref{eqn:1-Parameter} is satisfied. A similar argument shows that \eqref{eqn:1-Parameter} is satisfied if $A_1x=0$ or $A_2x=0$. Thus we can assume that $A_0x\ne0$, $A_1x\ne0$ and $A_2x\ne0$. By carrying out the manipulation $$\bigl((\l_0A_0+\l_1A_1+\l_2A_2)x\bigl)\otimes(A_2x)-(A_2x)\otimes\bigl((\l_0A_0+\l_1A_1+\l_2A_2)x\bigl)=0,$$
and the same with subscripts 1 and 2 exchanged and 0 and 2 exchanged, we get
$$
(\l_0\D_1-\l_1\Delta_0) x^{\otimes2}=0,\qquad (\l_0\Delta_2-\l_2\D_0) x^{\otimes2}=0,\qquad (\l_1\Delta_2-\l_2\D_1) x^{\otimes2}=0,
$$
which gives \eqref{eqn:1-Parameter} with $z= x^{\otimes2}\ne0$.

Next, we prove that \eqref{eqn:1-Parameter}$\Rightarrow$\eqref{eqn:2-Parameter}. We prove that if one of the MPPs in \eqref{eqn:1-Parameter} has a solution, say the first, then they all have a simultaneous solution such that \eqref{eqn:2-Parameter} is satisfied. Suppose that the first pencil in \eqref{eqn:1-Parameter} has a solution so that 
$$
\big(\mu_0(A_0\!\otimes\! A_2\!-\!A_2\!\otimes\! A_0)\!+\!\mu_1(A_1\!\otimes\! A_2\!-\!A_2\!\otimes\! A_1)\big)z\!=\!0,\quad\left[\!\!\begin{array}{c}\mu_0\\\mu_1\end{array}\!\!\right]\!\!\ne\!0,\quad z\!\ne\!0,
$$
which can be written as
$$
\big((\mu_0A_0+\mu_1A_1)\!\otimes\! A_2\!-\!A_2\!\otimes\!(\mu_0A_0+\mu_1A_1) \big)z\!=\!0,\quad\left[\!\!\begin{array}{c}\mu_0\\\mu_1\end{array}\!\!\right]\!\!\ne\!0,\quad z\!\ne\!0.
$$
It follows from Theorem~\ref{thm:m=2-square} (with $A=\mu_0A_0+\mu_1A_1$ and $B=A_2$) that there exists $\nu_1,\nu_2$, not both zero, and $x\ne0$ such that
$$
\big(\nu_1(\mu_0A_0+\mu_1A_1)+\nu_2A_2\big)x=0.
$$
Since $\nu_1,\nu_2$ are not both zero and $\mu_0,\mu_1$ are not both zero then $\l_0\!:=\!\nu_1\mu_0,~\l_1\!:=\!\nu_1\mu_1,~\l_2\!:=\!\nu_2$ 
are not all zero and this proves \eqref{eqn:2-Parameter}. Furthermore, since \eqref{eqn:2-Parameter} is satisfied, the proof of \eqref{eqn:2-Parameter}$\Rightarrow$\eqref{eqn:1-Parameter} shows that \eqref{eqn:1-Parameter} have a simultaneous solution $(\l_0,\l_1,\l_2)$ and a strongly decomposable $z\!=\!x^{\otimes2}$. This proves \eqref{eqn:1-Parameter}$\Rightarrow$\eqref{eqn:2-Parameter}.
\end{proof}

\begin{remark}\label{rem:3.2significant}
    The significance of Theorem~\ref{thm:m=2-square} (specifically that the null space of the Kronecker commutator always has a strongly decomposable element) is that it allows us to capture all (and only) the solutions of the two-parameter MPP in \eqref{eqn:1-Parameter}. In \cite{Alsubaie2019}, a regularity assumption was needed to show that \eqref{eqn:1-Parameter} captures only the solutions of the two-parameter MPP Problem~\ref{prob:m=2}.
\end{remark}

\section{The Kronecker Commutator and Anti-Commutator Operators}\label{sec:Commutator}
Since the MPPs in Theorem~\ref{thm:m=2-square} and Theorem~\ref{thm:1-Parameter} are given in terms of Kronecker commutator operators of the form $\D=A\otimes B-B\otimes A$ for $A,B\in\CC^{{m}\times{n}}$, we give a detailed investigation of their properties in this section. This will be used in Section~\ref{sec:Deflation} to highlight further properties of the one-parameter MPPs in Theorem~\ref{thm:1-Parameter} and simplify the computation of their solution.

We start with the definition of the commutation matrix and the symmetric and skew-symmetric projection matrices. The following lemma gives the definition and presents some of their properties needed in our work.

\begin{lemma}\label{lem:commutation}
Let  $\K^{{n}}$ denote the commutation matrix (first introduced in a different form in \cite{Tracy1969}, re-named and developed further in \cite{macrae1974} and re-named again and defined explicitly in the current form in  \cite{Magnus1979,Magnus1980}) as
$$
\K^{n}\!=\!\sum_{i=1}^{n}\sum_{j=1}^{n} (e^{n}_i{e^{n}_j}^{T})\otimes (e^{n}_j{e^{n}_i}^{T})\!=\!\sum_{i=1}^{n}{e^{n}_i}^{T}\otimes I_{n}\otimes e^{n}_i\!=\!\sum_{i=1}^{n}e^{n}_i\otimes I_{n}\otimes {e^{n}_i}^{T}\!\in\!\RR^{{n}^2\times {n}^2},
$$
where $e^{n}_i$ denotes the $i$th column of $I_{n}$.
Define the symmetric and skew-symmetric projection matrices as  \be\label{eqn:HnFn}
{\cH^{{n}}}:=\frac{1}{2}(I_{{n}^2}+\K^{{n}})\in\RR^{{n}^2\times {n}^2},\qquad{\F^{{n}}}:=\frac{1}{2}(I_{{n}^2}-\K^{{n}}) \in\RR^{{n}^2\times {n}^2},
\ee
respectively. Then the following properties can be readily established from the definitions and simple manipulations:
\begin{enumerate}[series=prop,label=P.\arabic*,itemsep=3mm,topsep=3mm]
\item $\quad\left[\begin{array}{ccc}\K^{n}&{\cH^{n}}&{\F^{n}}\end{array}\right]=\left[\begin{array}{ccc}{\K^{n}}^T&{\cH^{n}}^T&{\F^{n}}^T\end{array}\right]\!\cdot$\label{itm:P1}
% \item $\quad {\K^{n}}^T=\K^{n},\ {\cH^{n}}^T={\cH^{n}},\ {\F^{n}}^T={\F^{n}}$.\label{itm:P1}
\item $\quad \left[\!\!\!\begin{array}{c}\K^{n}\\{\cH^{n}}\\{\F^{n}}\end{array}\!\!\!\right]\!\!\left[\!\!\!\begin{array}{ccc}\K^{n}&{\cH^{n}}&{\F^{n}}\end{array}\!\!\!\right]\!=\!\left[\!\!\!\begin{array}{ccc}I_{{n}^2}&{\cH^{n}}\K^{n}&{\F^{n}}\K^{n}\\
\K^{n}{\cH^{n}}&{\cH^{n}}&0\\
\K^{n}{\F^{n}}&0&{\F^{n}}\end{array}\!\!\!\right]\!=\!\left[\!\!\!\begin{array}{ccc}I_{{n}^2}&{\cH^{n}}&-{\F^{n}}\\
{\cH^{n}}&{\cH^{n}}&0\\
-{\F^{n}}&0&{\F^{n}}\end{array}\!\!\!\right]\!\!\cdot$\label{itm:P2}
% \item $\quad \K^{{n}}^2=I,\ {\cH^{{n}}}^2={\cH^{{n}}},\ {\F^{{n}}}^2={\F^{{n}}}$.
% \item $\quad {\cH^{{n}}}{\F^{{n}}}={\F^{{n}}}{\cH^{{n}}}=0$.
\item $\quad {\cH^{{n}}} + {\F^{{n}}} =({\cH^{{n}}})^2 + ({\F^{{n}}})^2 = I_{{n}^2}$.\label{itm:P3}
\end{enumerate}
The next property gives implicit definitions of the commutation, symmetric and skew-symmetric projection matrices in term of their action on $\vec(Z)$, for any $Z\in\CC^{{n}\times {n}}$:
\begin{enumerate}[resume*=prop]
\item $\quad \left[\begin{array}{c}\K^{n}\\{\cH^{n}}\\{\F^{n}}\end{array}\right]\vec(Z)=\left[\begin{array}c\vec(Z^T)\\\frac{1}{2}\vec(Z+Z^T)\\\frac{1}{2}\vec(Z-Z^T)\end{array}\right]\!\!\cdot$\label{itm:P4}\\
% \item $\quad \K^{{n}} \vec(X)=\vec(X^T)$,
% \item $\quad {\cH^{{n}}}\vec(X)=\ \frac{1}{2}\vec(X+X^T)$
% \item $\quad {\F^{{n}}}\vec(X)=\frac{1}{2}\vec(X-X^T)$
\end{enumerate}
The following properties relate to the action on Kronecker products of any $A, B\in\CC^{{m}\times{n}}$, the Kronecker commutator operator $\D:=A\otimes B-B\otimes A\in\CC^{{m}^2\times {n}^2}$ and the Kronecker anti-commutator operator $\Dt:=A\otimes B+B\otimes A\in\CC^{{m}^2\times {n}^2}$:
\begin{enumerate}[resume*=prop]
\item $\quad \K^{m}(A\otimes B)\left[\begin{array}{cc}\K^{n}&I_{{n}^2}\end{array}\right]=(B\otimes A)\left[\begin{array}{cc}I_{{n}^2}&\K^{n}\end{array}\right]\!\cdot$\label{itm:P5}
    % \item $\quad \K^{M}(A\otimes B)=(B\otimes A)\K^{{n}}$, $(A\otimes B)\K^{{n}}=\K^{M}(B\otimes A)$, $\K^{M}(A\otimes B)\K^{{n}}=B\otimes A$\label{itm:P5}
    \item $\quad\left[\begin{array}{c}\K^{m}\D\\{\cH^{m}}\D\\{\F^{m}}\D\end{array}\right]=\left[\begin{array}{r}-\D\K^{n}\\\D {\F^{n}}\\\D{\cH^{n}}\end{array}\right]\!\!,
    \qquad\left[\begin{array}{c}\K^{m}\Dt\\{\cH^{m}}\Dt\\{\F^{m}}\Dt\end{array}\right]=\left[\begin{array}{c}\Dt\K^{n}\\\Dt {\cH^{n}}\\\Dt{\F^{n}}\end{array}\right]\!\!\cdot$\label{itm:P6}
    % \item $\quad \K^{m}{\D}\!=\!-{\D}\K^{n}.$\label{itm:P6}
        \item $\quad\left[\!\!\!\begin{array}{c}\K^{m}\\{\cH^{m}}\\{\F^{m}}\end{array}\!\!\!\right]\!\!\D\!\left[\!\!\!\begin{array}{ccc}\K^{n}&\!\!\!{\cH^{n}}&\!\!\!{\F^{n}}\end{array}\!\!\!\right]\!\!=\!\!\left[\!\!\!\begin{array}{ccc}-\D&-{\F^{m}}\D\K^{n}&-{\cH^{m}}\D\K^{n}\\
    -\K^{m}\D{\F^{n}}&0&2{\cH^{m}}(A\otimes B){\F^{n}}\\
    -\K^{m}\D{\cH^{n}}&2{\F^{m}}(A\otimes B){\cH^{n}}&0\end{array}\!\!\!\right]\!\!\cdot$
    
    $\quad\left[\!\!\!\begin{array}{c}\K^{m}\\{\cH^{m}}\\{\F^{m}}\end{array}\!\!\!\right]\!\!\Dt\!\left[\!\!\!\begin{array}{ccc}\K^{n}&\!\!\!{\cH^{n}}&\!\!\!{\F^{n}}\end{array}\!\!\!\right]\!\!=\!\!\left[\!\!\!\begin{array}{ccc}\Dt&{\cH^{m}}\Dt\K^{n}&{\F^{m}}\Dt\K^{n}\\
    \K^{m}\Dt{\cH^{n}}&2{\cH^{m}}(A\otimes B){\cH^{n}}&0\\
    \K^{m}\Dt{\F^{n}}&0&2{\F^{m}}(A\otimes B){\F^{n}}\end{array}\!\!\!\right]\!\!\cdot$
    \label{itm:P7}
    % \item $\quad \K^{m}{\D}\K^{n}\!=\!-{\D}$.\label{itm:P7}
    % \item $\quad {\cH^{m}}{\D}{\cH^{{n}}}=0$\label{itm:wtf9}
    % \item $\quad {\F^{m}}{\D}{\F^{{n}}}=0$\label{itm:wtf10}
    % \item $\quad {\F^{m}}{\D}{\cH^{{n}}}=2{\F^{m}}(A\otimes B){\cH^{{n}}}$\label{itm:wtf11}
    % \item $\quad {\cH^{m}}{\D}{\F^{{n}}}=2{\cH^{m}}(A\otimes B){\F^{{n}}}$\label{itm:wtf11a}
    % \item $\quad {\F^{m}}{\D} = {\D}{\cH^{{n}}}$ \label{itm:wtf12}
    % \item $\quad {\cH^{m}}{\D} = {\D}{\F^{{n}}}$ \label{itm:wtf13}
\end{enumerate}
Furthermore, suppose that $z=\vec(Z)\in\CC^{{n}^2}$ for any $Z\in\CC^{{n}\times {n}}$. Then
\be\label{eqn:sym-skew}
z=({\cH^{n}}+{\F^{n}})z={\cH^{n}}z+{\F^{n}}z=z_\mathrm{sym}+z_{\mathrm{skew}},
\ee
where,
\be\label{eqn:sym-skew-expressions}
z_\mathrm{sym}={\cH^{n}}z=\frac{1}{2}\vec(Z+Z^T),\qquad
z_\mathrm{skew}={\F^{n}}z=\frac{1}{2}\vec(Z-Z^T),
\ee
are the symmetric and skew-symmetric parts of $z$, respectively. Finally, if $z\in\CC^{n^2}$ is in the null space of $\D$, then so are $z_\mathrm{sym}$ and $z_\mathrm{skew}$.
\end{lemma}
\begin{proof}
Properties~\ref{itm:P1}-\ref{itm:P3} follow from the definitions and simple manipulations and Property~\ref{itm:P4} gives the implicit definitions. Property~\ref{itm:P5} is proved in \cite{Magnus1979,Magnus1980} and gives $\K^{n}$ its name. Properties~\ref{itm:P6} and \ref{itm:P7} follow from Property~\ref{itm:P5} and the definitions and simple manipulations. The proof of \eqref{eqn:sym-skew} and the expressions in \eqref{eqn:sym-skew-expressions} follow from Property~\ref{itm:P3} and the definitions in \eqref{eqn:HnFn}. Finally, suppose that $\D z=0$. Then
\bean
\D z=0&\Rightarrow& {\cH^{n}}\D z=0,\quad {\F^{n}}\D z=0,\\
&\Rightarrow& \D{\F^{n}}z=0,\quad \D{\cH^{n}}z=0,\quad({\rm from~\ref{itm:P6}})
\eean
which proves that $\D z_\mathrm{sym}=0$ and $\D z_\mathrm{skew}=0$ from the definitions in \eqref{eqn:sym-skew-expressions}.
\end{proof}

For illustration, we give the commutation, symmetric and skew-symmetric projection matrices for ${n}=3$ with the entries of $Z$ being the labels for clarity:
\bean
\K^3\!=\!\!\!\!\!\!
\left.\begin{array}{rl}
        &
        \left.\begin{array}{ccccccccc}
        \!\!\!\!\!\!\!{\scriptstyle11}&\!\!{\scriptstyle21}&\!\!{\scriptstyle31}&\!\!{\scriptstyle12}&\!{\scriptstyle22}&\!\!{\scriptstyle32}&\!\!{\scriptstyle13}&\!{\scriptstyle23}&\!\!{\scriptstyle33}\end{array}\right.\\
        \left.\begin{array}{c}
        {\scriptstyle11}\\
        {\scriptstyle21}\\
        {\scriptstyle31}\\
        {\scriptstyle12}\\
        {\scriptstyle22}\\
        {\scriptstyle32}\\
        {\scriptstyle13}\\
        {\scriptstyle23}\\
        {\scriptstyle33}
        \end{array}\right.&
\!\!\!\!\!\!\!\!\left[\!\begin{array}{ccccccccc}
     1  &   0  &   0 &    0  &   0 &    0   &  0 &    0 &    0\\
     0   &  0  &   0  &   1&     0  &   0   &  0  &   0  &   0\\
     0&     0&     0&     0&     0&     0&     1&     0&     0\\
     0&     1&     0&     0&     0&     0&     0&     0&     0\\
     0&     0&     0&     0&     1&     0&     0&     0&     0\\
     0&     0&     0&     0&     0&     0&     0&     1&     0\\
     0&     0&     1&     0&     0&     0&     0&     0&     0\\
     0&     0&     0&     0&     0&     1&     0&     0&     0\\
     0&     0&     0&     0&     0&     0&     0&     0&     1
     \end{array}\!\right]\!\!
     \end{array}\right.\!\!\!\!,~\cH^3\!=\!\!\!\!\!\!\!
     \left.\begin{array}{rl}
        &
        \left.\begin{array}{ccccccccc}
        \!\!\!\!\!\!\!\!\!{\scriptstyle11}&\!{\scriptstyle21}&\!{\scriptstyle31}&\!{\scriptstyle12}&\!\!{\scriptstyle22}&\!{\scriptstyle32}&\!{\scriptstyle13}&\!{\scriptstyle23}&\!{\scriptstyle33}\end{array}\right.\\
        \left.\begin{array}{c}
        {\scriptstyle11}\\
        {\scriptstyle21}\\
        {\scriptstyle31}\\
        {\scriptstyle12}\\
        {\scriptstyle22}\\
        {\scriptstyle32}\\
        {\scriptstyle13}\\
        {\scriptstyle23}\\
        {\scriptstyle33}
        \end{array}\right.&
\!\!\!\!\!\!\!\!\!\!\left[\!\begin{array}{ccccccccc}
         1&         0&         0&         0&         0&         0&         0&         0&         0\\
         0&    \frac{1}{2}&         0&    \frac{1}{2}&         0&         0&         0&         0&         0\\
         0&         0&    \frac{1}{2}&         0&         0&         0&    \frac{1}{2}&         0&         0\\
         0&    \frac{1}{2}&         0&    \frac{1}{2}&         0&         0&         0&         0&         0\\
         0&         0&         0&         0&    1&         0&         0&         0&         0\\
         0&         0&         0&         0&         0&    \frac{1}{2}&         0&    \frac{1}{2}&         0\\
         0&         0&    \frac{1}{2}&         0&         0&         0&    \frac{1}{2}&         0&         0\\
         0&         0&         0&         0&         0&    \frac{1}{2}&         0&    \frac{1}{2}&         0\\
         0&         0&         0&         0&         0&         0&         0&         0&    1
         \end{array}\!\right]\!\!
     \end{array}\right.\!\!\!\!,
  \eean
  \bean
     \F^3=\!\!\!\!\!\!\!
     \left.\begin{array}{rl}
        &
        \left.\begin{array}{ccccccccc}
        \!\!\!\!\!\!\!\!{\scriptstyle11}&~~\!{\scriptstyle21}&~~{\scriptstyle31}&~~{\scriptstyle12}&~~\!{\scriptstyle22}&~~\!{\scriptstyle32}&~~{\scriptstyle13}&~~{\scriptstyle23}&~~\!{\scriptstyle33}\end{array}\right.\\
        \left.\begin{array}{c}
        {\scriptstyle11}\\
        {\scriptstyle21}\\
        {\scriptstyle31}\\
        {\scriptstyle12}\\
        {\scriptstyle22}\\
        {\scriptstyle32}\\
        {\scriptstyle13}\\
        {\scriptstyle23}\\
        {\scriptstyle33}
        \end{array}\right.&
\!\!\!\!\!\!\!\!\!\!\left[\begin{array}{rrrrrrrrr}
         0&         0&         0&         0&         0&         0&         0&         0&         0\\
         0&    \frac{1}{2}&         0&   -\frac{1}{2}&         0&         0&         0&         0&         0\\
         0&         0&    \frac{1}{2}&         0&         0&         0&   -\frac{1}{2}&         0&         0\\
         0&   -\frac{1}{2}&         0&    \frac{1}{2}&         0&         0&         0&         0&         0\\
         0&         0&         0&         0&         \phantom{-}0&         0&         0&         0&         0\\
         0&         0&         0&         0&         0&    \frac{1}{2}&         0&   -\frac{1}{2}&         0\\
         0&         0&   -\frac{1}{2}&         0&         0&         0&    \frac{1}{2}&         0&         0\\
         0&         0&         0&         0&         0&   -\frac{1}{2}&         0&    \frac{1}{2}&         0\\
         0&         0&         0&         0&         0&         0&         0&         0&         \phantom{-}0
         \end{array}\right]\!\!
     \end{array}\right.\!\!\!\!,\quad Z=\left[\begin{array}{ccc}
11&12&13\\
21&22&23\\
31&32&33\end{array}\right]\!\!\cdot
\eean
The column and row labels denote the labels of the entries of $Z$ in $\vec(Z)$.

Next, we define three selection and two associated matrices. We also establish a comprehensive list of their properties, some of which we will be using and others for completeness since they may be of general interest.
\begin{definition}\label{def:diagonal-lower-upper}
Let $Z\in\CC^{{n}\times {n}}$ and let $z=\vec(Z)\in\CC^{{n}^2}$. Let $d(Z)\in\CC^{n}$ and $l(Z)\in\CC^{{n}({n}-1)/2}$ denote the vectors obtained from $z$ by keeping only the diagonal and strictly lower triangular elements of $Z$, in the same order they occur in $z$, respectively, and let $u(Z)\in\CC^{{n}({n}-1)/2}$ be the vector obtained from $z$ by keeping only the strictly upper triangular elements of $Z$, in the same order they occur in $\K^{n} z$.
\end{definition}

\begin{lemma}\label{lem:Selection}
Let the diagonal, lower triangular and upper triangular selection matrices $\S^{n}_D\!\in\!\RR^{{n}\times {n}^2}$, $\S^{n}_L\!\in\!\RR^{\frac{{n}({n}-1)}{2}\times {n}^2}$ and $\S^{n}_U\!\in\!\RR^{\frac{{n}({n}-1)}{2}\times {n}^2}$ be defined such that
\begin{enumerate}[resume*=prop]
\item $\quad\left[\begin{array}{c}\S^{n}_D\\\S^{n}_L\\\S^{n}_U\end{array}\right]\vec(Z)=\left[\begin{array}{c}d(Z)\\l(Z)\\u(Z)\end{array}\right]\!\!\cdot$\label{itm:P8}
\end{enumerate}
Let $e^p_i$ denote the $i$-th column of $I_p$ and for each $1\!\le\! j\!<\!i\!\le\! {n}$ let $k_{ij}\!=\!(j\!-\!1){n}\!+\!i\!-\!\frac{j(j\!+\!1)}{2}$. Explicit definitions for the selection matrices are as follows:
\bean
\S^{n}_D&=&\sum_{i=1}^{n}e^{n}_i\otimes{e^{n}_i}^{T}\otimes {e^{n}_i}^{T}\!\!,\\
\S^{n}_L&=&\sum_{j=1}^{{n}-1}\sum_{i=j+1}^{n}e^{{n}({n}\!-\!1)/2}_{k_{ij}}\otimes {e^{n}_j}^{T}\otimes {e^{n}_i}^{T}\!\!,\\
\S^{n}_U&=&\sum_{j=1}^{{n}-1}\sum_{i=j+1}^{n}e^{{n}({n}\!-\!1)/2}_{k_{ij}}\otimes {e^{n}_i}^{T}\otimes {e^{n}_j}^{T}\!.
\eean
Then, the following properties are satisfied:
\begin{enumerate}[resume*=prop]
    \item \label{itm:P9}$\quad P^{n}:=\!\left[\!\!\!\begin{array}{c}\S^{n}_D\\\S^{n}_L\\\S^{n}_U\end{array}\!\!\!\right]\!$ is a permutation matrix satisfying
\bean
{P^{n}}^TP^{n}&=&{\S^{n}_D}^{\!T}\!\S^{n}_D\!+\!{\S^{n}_L}^{\!T}\!\S^{n}_L\!+\!{\S^{n}_U}^{\!T}\!\S^{n}_U\!=\!I_{{n}^2},\\
P^{n}{P^{n}}^T&=&\left[\!\!\!\begin{array}{c}\S^{n}_D\\\S^{n}_L\\\S^{n}_U\end{array}\!\!\!\right]\left[\begin{array}{ccc}{\S^{n}_D}^{\!T}\!&{\S^{n}_L}^{\!T}\!&{\S^{n}_U}^{\!T}\!\end{array}\right]=\left[\begin{array}{ccc}I_{n}&0&0\\0&I_{{n}({n}-1)/2}&0\\0&0&I_{{n}({n}-1)/2}\end{array}\right]\!\!.
\eean
% \item $\quad \S^{n}_D{\S^{n}_D}^{\!T}\!\! = I_{n}$ and $\S^{n}_L{\S^{n}_L}^{\!T}\! = \S^{n}_U{\S^{n}_U}^{\!T}\! = I_{{n}(N\!-\!1)/2}$
    \item \label{itm:P10}$\quad {\S^{n}_D}^{\!T}\!\S^{n}_D = I_D$, ${\S^{n}_L}^{\!T}\!\S^{n}_L = I_L$ and ${\S^{n}_U}^{\!T}\!\S^{n}_U = I_U$ where $I_D, I_L$ and $I_U$ are ${n}^2\times {n}^2$ diagonal matrices with zeros on the diagonals except for ones at the locations corresponding to the diagonal, strictly lower triangular and strictly upper triangular elements in $\vec(Z)$, respectively.
    \item \label{itm:P11}$\quad \left[\begin{array}{c}\S^{n}_D\\\S^{n}_L\\\S^{n}_U\end{array}\right]\K^{n}=\left[\begin{array}{c}\S^{n}_D\\\S^{n}_U\\\S^{n}_L\end{array}\right]\!\!\cdot$
    \item \label{itm:P12}$\quad \K^{n} ={\S^{n}_D}^{\!T}\!\S^{n}_D + {\S^{n}_L}^{\!T}\!\S^{n}_U + {\S^{n}_U}^{\!T}\!\S^{n}_L.$
    \item \label{itm:P13}$\quad \S^{n}_D\left[\begin{array}{cc}{\cH^{n}}&{\F^{n}}\end{array}\right] = \S^{n}_D\left[\begin{array}{cc}I_{{n}^2}&0\end{array}\right]\!$.
    \item \label{itm:P14}$\quad \S^{n}_L\left[\begin{array}{cc}{\cH^{n}}&{\F^{n}}\end{array}\right]=\S^{n}_U\left[\begin{array}{cc}{\cH^{n}}&-{\F^{n}}\end{array}\right]=\frac{1}{2}\left[\begin{array}{cc}\S^{n}_L+\S^{n}_U&\S^{n}_L-\S^{n}_U\end{array}\right]\!.$
    \end{enumerate}
Furthermore, define the associated matrices
\be\label{eqn:SnTn}
\S^{n}\!:= \!\!\matTwoTwo{\!\!\matTwoOne{\S^{n}_D}{\!\!\sqrt{2}S^{n}_L}\!\!}{\!\!0\!\!}{\!\!0\!\!}{\!\!\sqrt{2}\S^{n}_U\!\!}\!\!\in\!\RR^{{n}^2\times2n^2}\!\!\!,
\quad
\T^{{n}} \!:= \!\!\left[\!\!\begin{array}{c}\S^{n}_D\\\frac{1}{\sqrt{2}}(\S^{n}_L\!+\!\S^{n}_U)\\\hline\frac{1}{\sqrt{2}}(\S^{n}_L\!-\!\S^{n}_U)\end{array}\!\!\right]\!\!=:\!\!\left[\!\!\begin{array}{c}\V^n\\\hline\U^n\end{array}\!\!\right]\!\!\in\!\RR^{{n}^2\times {n}^2}\!\!\!.
\ee
Then the following properties are satisfied:
\begin{enumerate}[resume*=prop]
    \item $\quad \T^{n}=\left[\begin{array}{c}\V^n\\\hline\U^n\end{array}\right]=\S^{n}\left[\begin{array}{c}{\cH^{n}}\\-{\F^{n}}\end{array}\right]\!=\left[\begin{array}{c}\left[\begin{array}{c}\S^{n}_D\\\sqrt{2}\S^{n}_L\end{array}\right]{\cH^{n}}\\\hline-\sqrt{2}\S^{n}_U{\F^{n}}\end{array}\right]\!\!\cdot$\label{itm:P15}
    \item $\quad \T^{n}{\T^{n}}^T={\T^{n}}^T\T^{n}=I_{{n}^2}$\label{itm:P16}.
\end{enumerate}
\end{lemma}
\begin{proof}
Property~\ref{itm:P8} expresses the implicit definitions of the selection matrices. The explicit definitions of $\S^{n}_D$, $\S^{n}_L$ and $\S^{n}_U$ are slight modifications to those of the elimination matrix in \cite{Magnus1980} which selects the lower triangular (rather than the strictly lower) part of $Z$. Property~\ref{itm:P9} follows since $P^{n}z$ is simply a rearrangement of the elements of $z$. We prove the second statement of Property~\ref{itm:P10} as the proof of the others is similar. Starting from the explicit definition of $\S^{n}_L$:
\begin{eqnarray*}
    {\S^{n}_L}^{T}S^{n}_L &=&
    \sum_{q=1}^{{n}-1}
    \sum_{p=q+1}^{n}
    \sum_{j=1}^{{n}-1}
    \sum_{i=j+1}^{n}\!
    \left(\!
        \left(e^{{n}({n}\!-\!1)/2}_{k_{pq}}\right)^{T}
        \!\otimes e^{n}_q
        \otimes e^{n}_p
    \right)\!
    \left(
        e^{{n}({n}\!-\!1)/2}_{k_{ij}}
        \otimes {e^{n}_j}^{T}\!
        \otimes {e^{n}_i}^{T}
    \right)\\
    &=&
    \sum_{q=1}^{{n}-1}
    \sum_{p=q+1}^{n}
    \sum_{j=1}^{{n}-1}
    \sum_{i=j+1}^{n}
    \!\left(
        \left(e^{{n}({n}\!-\!1)/2}_{k_{pq}}\right)^{T}\! e^{{n}({n}\!-\!1)/2}_{k_{ij}}
    \right)
    \otimes
    \left(
        e^{n}_q \otimes e^{n}_p
    \right)\!\!
    \left(
        {e^{n}_j}^{T} \!\!\otimes {e^{n}_i}^{T}
    \right).
\end{eqnarray*}
The first term in the summation is a scalar and is equal to $1$ if and only if $k_{pq}=k_{ij}$, and zero otherwise. Setting $p=i$ and $q=j$ we get
\begin{eqnarray*}
    {\S^{n}_L}^{T}\!S^{n}_L &=&
    \sum_{j=1}^{{n}-1}
    \sum_{i=j+1}^{n}
    \!\left(
        e^{n}_j \otimes e^{n}_i
    \right)\!
    \left(
        {e^{n}_j}^{T} \!\otimes {e^{n}_i}^{T}
    \right)    =
    \sum_{j=1}^{{n}-1}
    \sum_{i=j+1}^{n}
        \!\left(e^{{n}^2}_{k_{ij}}\right)
        \!\left({e^{{n}^2}_{k_{ij}}}\right)^{T}\!,
\end{eqnarray*}
which is a diagonal matrix with ones at the $k_{ij}$ locations along the diagonal, and zero otherwise, which proves the property. Property~\ref{itm:P11} states that the diagonal elements of $Z$ and $Z^T$ are the same and that the strictly lower (resp., upper) triangular elements of $Z$ are the same as the strictly upper (resp., lower) triangular elements of $Z^T$. Property~\ref{itm:P12} (which follows by pre-multiplying Property~\ref{itm:P11} by $\left[\!\!\begin{array}{ccc}{\S^{n}_D}^{\!T}&{\S^{n}_L}^{\!T}&{\S^{n}_U}^{\!T}\end{array}\!\!\right]$ and using Property~\ref{itm:P9}) 
shows that $\K^{n}$ can be defined in terms of the selection matrices: to transpose a square matrix, we swap the lower and upper triangular parts. That is, we select the different parts, swap the lower and upper parts and then undo the permutation, so that,
\bean
    \K^{n}\!=\!\!\matOneThr{{\S^{n}_D}^{\!T}\!\!}{{\S^{n}_L}^{\!T}\!\!}{\!\!{\S^{n}_U}^{\!T}\!\!}\!\!\!\matThrThr{I_{n}}{0}{0}{0}{0}{I_{{n}({n}\!-\!1)/2}}{0}{I_{{n}({n}\!-\!1)/2}}{0}
    \!\!\!\!\matThrOne{\S^{n}_D}{\S^{n}_L}{\S^{n}_U}\!\!\cdot
\eean
Property~\ref{itm:P13} states that the diagonal elements of $Z$ and $\frac{1}{2}(Z+Z^T)$ are the same and that the diagonal elements of $\frac{1}{2}(Z-Z^T)$ are zero. Property~\ref{itm:P14} states that the lower and upper triangular elements of $\frac{1}{2}(Z+Z^T)$ are the same while those of $\frac{1}{2}(Z-Z^T)$ are the negatives of each other. Property~\ref{itm:P15} follows from Properties~\ref{itm:P13} and \ref{itm:P14} and direct evaluation. Finally, Property~\ref{itm:P16} follows from Property~\ref{itm:P9} and a direct evaluation using the expression for $\T^{n}$ in \eqref{eqn:SnTn}. 
\end{proof}

For illustration, we give the selection matrices for ${n}=4$ with the entries of $Z$ being the labels for clarity:
$$
Z=\left[\begin{array}{cccc}
11&12&13&14\\
21&22&23&24\\
31&32&33&34\\
41&42&43&44\end{array}\right]\!\!,
$$
    \bean
        \left[\begin{array}{c}\S^{n}_D\\\hline \S^{n}_L\\\hline \S^{n}_U\end{array}\right] =
        \left.\begin{array}{rl}
        &
        \left.\begin{array}{cccccccccccccccc}
        \!\!\!\!\!\!\!{\scriptstyle11}&\!\!{\scriptstyle21}&\!\!{\scriptstyle31}&\!\!{\scriptstyle41}&\!\!{\scriptstyle12}&\!{\scriptstyle22}&\!\!{\scriptstyle32}&\!\!{\scriptstyle42}&\!\!{\scriptstyle13}&\!{\scriptstyle23}&\!\!{\scriptstyle33}&\!\!{\scriptstyle43}&\!{\scriptstyle14}&\!\!{\scriptstyle24}&\!\!{\scriptstyle34}&\!\!{\scriptstyle44}\end{array}\right.\\
        \left.\begin{array}{c}
        {\scriptstyle11}\\
        {\scriptstyle22}\\
        {\scriptstyle33}\\
        {\scriptstyle44}\\\hline
        {\scriptstyle21}\\
        {\scriptstyle31}\\
        {\scriptstyle41}\\
        {\scriptstyle32}\\
        {\scriptstyle42}\\
        {\scriptstyle43}\\\hline
        {\scriptstyle12}\\
        {\scriptstyle13}\\
        {\scriptstyle14}\\
        {\scriptstyle23}\\
        {\scriptstyle24}\\
        {\scriptstyle34}
        \end{array}\right.&
        \!\!\!\!\!\!\!\!\!\left[\begin{array}{cccccccccccccccc}
        1&0&0&0&0&0&0&0&0&0&0&0&0&0&0&0\\
        0&0&0&0&0&1&0&0&0&0&0&0&0&0&0&0\\
        0&0&0&0&0&0&0&0&0&0&1&0&0&0&0&0\\
        0&0&0&0&0&0&0&0&0&0&0&0&0&0&0&1\\\hline
        0&1&0&0&0&0&0&0&0&0&0&0&0&0&0&0\\
        0&0&1&0&0&0&0&0&0&0&0&0&0&0&0&0\\
        0&0&0&1&0&0&0&0&0&0&0&0&0&0&0&0\\
        0&0&0&0&0&0&1&0&0&0&0&0&0&0&0&0\\
        0&0&0&0&0&0&0&1&0&0&0&0&0&0&0&0\\
        0&0&0&0&0&0&0&0&0&0&0&1&0&0&0&0\\\hline
        0&0&0&0&1&0&0&0&0&0&0&0&0&0&0&0\\
        0&0&0&0&0&0&0&0&1&0&0&0&0&0&0&0\\
        0&0&0&0&0&0&0&0&0&0&0&0&1&0&0&0\\
        0&0&0&0&0&0&0&0&0&1&0&0&0&0&0&0\\
        0&0&0&0&0&0&0&0&0&0&0&0&0&1&0&0\\
        0&0&0&0&0&0&0&0&0&0&0&0&0&0&1&0
        \end{array}\right]
        \end{array}\right.\!\!\!\!\!\!\cdot
    \eean
The column labels denote the labels of the entries of $Z$ in $\vec(Z)$. The row labels denote the labels of the entries of $Z$ in $d(Z)$, $l(Z)$ and $u(Z)$, respectively. 
\begin{remark}
It is straightforward to use the examples above to illustrate the properties enumerated in Lemmas~\ref{lem:commutation} and \ref{lem:Selection}. Note that a combination of $\S^{n}_D$ and $S^{n}_L$ in different order was introduced in \cite{Magnus1980} as the
elimination matrix as it eliminates the strictly upper triangular part. For our purposes, we separate them and introduce $\S^{n}_U$.
\end{remark}

The matrices $\K^{n},~{\cH^{n}},~{\F^{n}},~\S^{n}_D,~\S^{n}_L$ and $\S^{n}_U$ give us a set of transformations that we can apply to vectors of the form $\vec{(Z)}$ in order to effect the manipulations on the matrix $Z$ equivalent to transposing, decomposing into symmetric and skew-symmetric parts, and selecting the diagonal, strictly lower and strictly upper triangular parts, respectively. Using these tools, we are ready to state the main result of this section.
\begin{theorem}\label{thm:Diagonalisation}
Let ${m}$ and ${n}$ be integers and let all variables be as defined in Lemma~\ref{lem:commutation} and Lemma~\ref{lem:Selection}. For any $A,B\!\in\!\CC^{{m}\times{n}}$, let $\D\!:=\!A\otimes B\!-\!B\otimes A\!\in\!\CC^{{m}^2\times {n}^2}$ and $\Dt:=A\otimes B+B\otimes A\in\CC^{{m}^2\times {n}^2}$ be the Kronecker commutator and anti-commutator operators, respectively. 
Then 
\begin{enumerate}
\item $\D$ can be block anti-diagonalized using orthogonal transformations as follows
\bean
\T^{m}\D{\T^{n}}^T\!=\!\left[\!\begin{array}{c}\V^m\\\U^m\end{array}\!\right]\!\D\!\left[\!\begin{array}{cc}{\V^n}^T&\!{\U^n}^T\end{array}\!\right]\!=
\!\!\!\!\!\!\!
\left.\begin{array}{rl}&\!\!\!\!\!\!\!\!\!\!\begin{tiny}\left.\begin{array}{cc}\frac{{n}({n}+1)}{2}&\!\!\!\!\!\frac{ {n}({n}-1)}{2}\end{array}\right.\end{tiny}\\
\begin{tiny}\left.\begin{array}{r}\frac{{m}({m}+1)}{2}\\\frac{{m}({m}-1)}{2}\end{array}\right.\end{tiny}&
\!\!\!\!\!\!\!\!\left[\begin{array}{cc}0&\D_{12}\\\D_{21}&0\end{array}\right]\end{array}\right.\!\!\!\!\!\!,
\eean
and $\Dt:=A\otimes B+B\otimes A\in\CC^{{m}^2\times {n}^2}$ can be block diagonalized using the same orthogonal transformations as follows:
\bean
\T^{m}\Dt{\T^{n}}^T\!=\!\left[\!\begin{array}{c}\V^m\\\U^m\end{array}\!\right]\!\Dt\!\left[\!\begin{array}{cc}{\V^n}^T&\!{\U^n}^T\end{array}\!\right]
\!=\!\!\!\!\!\!\!
\left.\begin{array}{rl}&\!\!\!\!\!\!\!\!\!\!\begin{tiny}\left.\begin{array}{cc}\frac{{n}({n}+1)}{2}&\!\!\!\!\!\frac{ {n}({n}-1)}{2}\end{array}\right.\end{tiny}\\
\begin{tiny}\left.\begin{array}{r}\frac{{m}({m}+1)}{2}\\\frac{{m}({m}-1)}{2}\end{array}\right.\end{tiny}&
\!\!\!\!\!\!\!\!\left[\begin{array}{cc}\Dt_{11}&0\\0&\Dt_{22}\end{array}\right]\end{array}\right.\!\!\!\!\!\!,
\eean
where the partitioning of $\T^m$ and $\T^n$ is defined in \eqref{eqn:SnTn} and where
\bean
\left[\!\!\!\begin{array}{cc}\Dt_{11}&\!\!\!\D_{12}\\\D_{21}&\!\!\!\Dt_{22}\end{array}\!\!\!\right]\!\!=\!\!
\left[\!\!\!\begin{array}{cc}\V^m\Dt {\V^n}^T&\!\!\V^m\D{\U^n}^T\\\U^m\D {\V^n}^T&\!\!\U^m\Dt{\U^n}^T\end{array}\!\!\!\right]\!\!=\!\!
\left[\!\!\!\begin{array}{cc}2\V^m(A\otimes B){\V^n}^T&\!\!2\V^m(A\otimes B)\,{\U^n}^T\\2\,\U^m(A\otimes B){\V^n}^T&\!\!2\,\U^m(A\otimes B)\,{\U^n}^T\end{array}\!\!\!\right]\!\!\cdot
\eean
\item Suppose that the null space of $\D_{21}$ is nonempty. Then it includes a vector 
\be\label{eqn:Unique Symmetric}
y=\V^nx^{\otimes2},
\ee
for some $0\ne x\in\CC^n$ with ${\V^n}^Ty=x^{\otimes2}$.
\item $\D$ has full column rank if and only if $\D_{21}$ has full column rank.
\end{enumerate}
\end{theorem}
\begin{proof}~
\begin{enumerate}
\item We will prove the result for $\D$ only as the proof for $\Dt$ is similar. The orthogonality of $\T^{m}$ and $\T^{n}$ follows from Property~\ref{itm:P16}. Using Properties~\ref{itm:P1}, \ref{itm:P7} and \ref{itm:P15} we have
    \begin{eqnarray*}
        \T^{{m}}\D{\T^{n}}^T\! &\!=\!& \S^{m}\!\matTwoOne{\!\!\!{\cH^{{m}}}\!\!\!}{\!\!\!-{\F^{{m}}}\!\!\!}\!\D\!\matOneTwo{\!\!\!{\cH^{{n}}}}{\!\!-{\F^{{n}}}\!\!\!}\!{\S^n}^T\!\!=\!
        \S^{m}\!\matTwoTwo{\!\!0}{-{\cH^{{m}}}\D{\F^{{n}}}\!\!}{\!\!-{\F^{{m}}}\D{\cH^{{n}}}}{0\!\!}\!S_{n}^T\\
        \!&\!=\!&\!\matTwoTwo{\!\!\!\!\matTwoOne{\!\!\S^{m}_D\!\!}{\!\!\sqrt{2}S^{m}_L\!\!}}{\!\!\!\!\!\!0}{0}{\!\!\!\!\!\!\sqrt{2}\S^{m}_U\!\!}\!\!\matTwoTwo{\!\!0}{\!\!\!\!\!\!-{\cH^{{m}}}\D{\F^{{n}}}\!\!}{\!\!-{\F^{{m}}}\D{\cH^{{n}}}}{\!\!\!\!\!\!0\!\!}\!\!\matTwoTwo{\!\!\matTwoOne{\!\!\S^{n}_D\!\!}{\!\!\sqrt{2}S^{n}_L}\!\!}{\!\!0\!\!}{\!\!0}{\!\!\sqrt{2}\S^{n}_U\!\!}^T\!\!\!\!,
    \end{eqnarray*}
which gives the result by a direct evaluation using \eqref{eqn:SnTn}. The equality of the two expressions for both $\D_{21}$ and $\D_{12}$ follows from Property~\ref{itm:P7}. 
\item Suppose that $\D_{21}$ loses column rank. Then $\D$ loses column rank and it follows from Theorem~\ref{thm:m=2-square} that there exists $0\!\ne\! x\!\in\!\CC^n$ such that $\D x^{\otimes2}\!=\!0$. Therefore,
\bea\nonumber
\D x^{\otimes2}\!=\!0&\Rightarrow& \T^m\D{\T^n}^T\T^n x^{\otimes2}\!=\!0\Rightarrow\!\left[\!\!\begin{array}{cc}0&\D_{12}\\\D_{21}&0\end{array}\!\!\right]\!\!\left[\!\!\begin{array}{c}\matTwoOne{\!\!\S^{n}_D\!\!}{\!\!\sqrt{2}S^{n}_L\!\!}\cH^n x^{\otimes2}\\-\sqrt{2}\S^{n}_U \F^nx^{\otimes2}\end{array}\!\!\right]\!\!=\!0\\\label{eqn:D21losesrank}
&\Rightarrow&\left[\!\!\begin{array}{cc}0&\!\!\!\D_{12}\\\D_{21}&\!\!\!0\end{array}\!\!\right]\!\!\left[\!\!\begin{array}{c}\V^n x^{\otimes2}\\0\end{array}\!\!\right]\!\!=\!0\Rightarrow \D_{21}\V^n x^{\otimes2}\!=\!0,
\eea
where we used the fact that since $x^{\otimes2}$ is symmetric, $\F^nx^{\otimes2}\!=\!0$. This proves the first equality in \eqref{eqn:Unique Symmetric} by defining $y\!=\!\V^nx^{\otimes2}\!\ne\!0$. The second equality follows by pre-multiplying the first by $ {\V^n}^T\!$ and the orthogonality of $\T^n$. 
\item It is clear that $\D_{21}$ has full column rank if $\D$ has full column rank. Let $\D_{21}$ have full column rank and suppose on the contrary that $\D$ loses column rank. It follows from Theorem~\ref{thm:m=2-square} there exists $0\!\ne\! x\!\in\!\CC^n$ such that $\D x^{\otimes2}\!=\!0$. This implies that $\D_{21}$ loses column rank from \eqref{eqn:D21losesrank}, thus proving the result.
\end{enumerate}
\end{proof}
\begin{remark}
While the permutation approach in \cite{Alsubaie2019} can be adapted to give a proof of the theorem, we have opted for the selection matrices approach as it is more specific to the two-parameter MPP and is more relevant to the familiar Kronecker commutator and anti-commutator operators. While we only use Kronecker commutator operators in this work, we have included the corresponding results for anti-commutator operators for completeness and since they may be of interest in other fields.
\end{remark}

\section{Deflation Scheme for the Two-Parameter MPP}\label{sec:Deflation}

In this section, we use the Kronecker commutator properties derived in Section~\ref{sec:Commutator} to present a deflation scheme for the one-parameter MPPs in \eqref{eqn:1-Parameter} by removing redundant equations and exploiting the special structure of the matrices and eigenvector.

Theorem~\ref{thm:Diagonalisation} shows that the rank properties of the $m^2\times n^2$ Kronecker commutator operator $\D$ are effectively captured by the rank properties of the $\frac{m(m-1)}{2}\times\frac{n(n+1)}{2}$ anti-diagonal block $\D_{21}$. The next result uses this to deflate the $m^2\times n^2$ one-parameter MPPs in Theorem~\ref{thm:1-Parameter} to equivalent $\frac{m(m-1)}{2}\times\frac{n(n+1)}{2}$ MPPs.
\begin{thm}\label{thm:1-Parameter-Compressed}Let all variables be as defined in Theorem~\ref{thm:1-Parameter} and Theorem~\ref{thm:Diagonalisation} and let $\tilde{m}=\frac{m(m-1)}{2}$ and $\tilde{n}=\frac{n(n+1)}{2}$. Define the Kronecker determinants
\begin{eqnarray}\label{Gammas}
    \G_0 = 2\,\U^m(A_1\!\otimes\! A_2){\V^n}^T = \U^m\D_0{\V^n}^T && \nonumber \\
    \G_1 = 2\,\U^m(A_2\!\otimes\! A_0){\V^n}^T = \U^m\D_1{\V^n}^T && \\
    \G_2 = 2\,\U^m(A_0\!\otimes\! A_1){\V^n}^T = \U^m\D_2{\V^n}^T && \nonumber
\end{eqnarray}
so that $\G_i\in\CC^{\tilde{m}\times\tilde{n}}$ for $i=0,1,2$. Then the ${m}^2\times {n}^2$ one-parameter MPPs in \eqref{eqn:1-Parameter} have a simultaneous solution (equivalently, the ${m}\times {n}$ two-parameter MPP in Problem~\ref{prob:m=2} has a solution) if and only if the following three $\tilde{m}\times\tilde{n}$ one-parameter MPPs
\be\label{eqn:1-Parameter-Compressed}
(\l_i\G_j\!-\!\l_j\G_i)y\!=\!0,~(i,j)\in\{(0,1),(0,2),(1,2)\},\quad0\!\ne\!\left[\!\!\begin{array}{c}\l_0\\\l_1\\\l_2\end{array}\!\!\right]\!\!\in\!\CC^3,\quad
 0\!\ne\! y\!\in\!\CC^{\tilde{n}}\!,
\ee
have a simultaneous solution.
\end{thm}
\begin{proof}
The matrix pencils in \eqref{eqn:1-Parameter} are equivalent to
$$
\T^{m}(\l_i\Delta_j\!-\!\l_j\Delta_i){\T^{n}}^T\!\T^{n}z\!=\!0,~(i,j)\in\{(0,1),(0,2),(1,2)\}.
$$
By Theorem~\ref{thm:Diagonalisation}, the matrices $\T^{m}\D_i{\T^{n}}^T$, $i=0,1,2$, have the form
$$
\T^{m}\D_i{\T^{n}}^T=\left[\begin{array}{cc}0&\star\\\G_i&0\end{array}\right]\!\!,\quad \G_i\in\CC^{\tilde{m}\times\tilde{n}},
$$
where $\star$ denotes terms whose expressions are not needed in the present context. Thus the matrix pencils in \eqref{eqn:1-Parameter} are equivalent to the following pencils:
$$
\left[\!\begin{array}{cc}0&\star\\\l_i\G_j\!-\!\l_j\G_i&0\end{array}\!\right]\!\T^{n}z=0,~(i,j)\in\{(0,1),(0,2),(1,2)\}.
$$
This shows that if \eqref{eqn:1-Parameter-Compressed} has a solution $0\ne\l=(\l_0,\l_1,\l_2)$ and $0\ne y\in\CC^{\tn}$, then \eqref{eqn:1-Parameter} has a solution $\l=(\l_0,\l_1,\l_2)$ and 
$$0\ne z={\T^n}^T\left[\begin{array}{c}y\\0\end{array}\right]={\V^n}^Ty\in\CC^{n^2}\!\!.$$
Suppose now that \eqref{eqn:1-Parameter} has a solution $0\ne\l=(\l_0,\l_1,\l_2)$ and $0\ne z\in\CC^{{n}^2}$. Then it follows from the last part of Theorem~\ref{thm:1-Parameter} that there exists a strongly decomposable eigenvector $0\ne x^{\otimes2}\in\CC^{{n}^2}$ satisfying \eqref{eqn:1-Parameter}. Now, Property~\ref{itm:P15} implies that
\bean
\T^{n}x^{\otimes2}&=&\left[\begin{array}{c}\left[\begin{array}{c}\S^{n}_D\\\sqrt{2}\S^{n}_L\end{array}\right]{\cH}^{n}\\-\sqrt{2}\S^{n}_U{\F}^{n}\end{array}\right]x^{\otimes2}
\\
&=&\left[\begin{array}{c}\left[\begin{array}{c}\S^{n}_D\\\sqrt{2}\S^{n}_L\end{array}\right]{\cH}^{n}x^{\otimes2}\\-\sqrt{2}\S^{n}_U{\F}^{n}x^{\otimes2}\end{array}\right]\\
&
=&\left[\begin{array}{c}\V^nx^{\otimes2}\\0\end{array}\right]=:\left[\begin{array}{c}y\\0\end{array}\right]\!\!,
\eean
since ${\F}^{n}x^{\otimes2}=0$ as $x^{\otimes2}$ is symmetric.
Therefore, for $(i,j)\in\{(0,1),(0,2),(1,2)\}$,
$$
\T^{m}(\l_i\D_j\!-\!\l_j\D_i){\T^{n}}^T\T^{n}x^{\otimes2}=\!\left[\!\begin{array}{cc}0&\star\\\l_i\G_j\!-\!\l_j\G_i&0\end{array}\!\right]\!\left[\!\begin{array}{c}y\\0\end{array}\!\!\right]\!=\!\left[\!\begin{array}{c}0\\(\l_i\G_j\!-\!\l_j\G_i)y\end{array}\!\right]\!=0,
$$
which proves that \eqref{eqn:1-Parameter-Compressed} has a solution $\l=(\l_0,\l_1,\l_2)$ and $0\ne y\in\CC^{\tn}$. It follows that the ${m}\times {n}$ two-parameter MPPs in \eqref{eqn:2-Parameter} have a solution if and only if the  $\tilde{m}\times\tilde{n}$ one-parameter MPPs in \eqref{eqn:1-Parameter-Compressed} have a solution.
\end{proof}
\section{The Solution of the Two-Parameter MPP when \texorpdfstring{$\boldsymbol{m=n+1}$}{m=n+1}}\label{sec:m=n+1}
The number of unknowns in Problem~\ref{prob:m=2} is $n+1$ and the number of equations is $m$ (see Remark~\ref{rem:Assumptions}). Thus when $m=n+1$, the number of equations is equal to the number of unknowns, and it might be expected that the problem always has at least one solution.   
The next result shows that this is the case. Furthermore, it shows that, under a certain rank condition on the Kronecker determinants, the set of linked matrix pencils in \eqref{eqn:1-Parameter-Compressed} can be decoupled and their solution reduces to finding the simultaneous eigenvectors of a set of three commuting matrices \cite{horn1985}.
\begin{thm}\label{thm:Commutativity}
With everything as defined in Theorem~\ref{thm:1-Parameter-Compressed}, suppose that ${m}\!=\!{n}\!+\!1$. Then the one-parameter MPPs in \eqref{eqn:1-Parameter-Compressed} (and therefore the $({n}\!+\!1)\times {n}$ two-parameter MPP in Problem~\ref{prob:m=2}) have at least one simultaneous solution, and generically $\tn$ solutions. Furthermore, suppose that there exist $\a_0,\a_1,\a_2\!\in\!\CC$ such that $\G\!:=\!\a_0\G_0\!+\!\a_1\G_1\!+\!\a_2\G_2$ is nonsingular. Then 
\begin{enumerate}
\item $\G^{-1}\G_0,~\G^{-1}\G_1$ and $\G^{-1}\G_2$ commute.
\item The solutions of the $({n}+1)\times {n}$ two-parameter MPP in Problem~\ref{prob:m=2} are given by the simultaneous solutions of the three $\tilde{n}\times\tilde{n}$ eigenvalue problems
\be\label{eqn:All solutions}
(\l_i I_{\tilde{n}}-\G^{-1}\G_i)y=0,\quad i=0,1,2,\quad 0\ne y\in\CC^{\tilde{n}}.
\ee
\end{enumerate}
\end{thm}
\begin{proof}
If ${m}={n}+1$, then $\tm=\tn$ and the Kronecker determinants in \eqref{eqn:1-Parameter-Compressed} are therefore square. Consider one of the pencils in equations \eqref{eqn:1-Parameter-Compressed}, say the first one: $(\l_0\G_1-\l_1\G_0)y=0$. We first show that this always has a nontrivial solution. If $\G_0$ is singular, then we can take $\l_0=0$, $\l_1$ an arbitrary complex number and $y\ne0$ can be chosen to be any vector in the null space of $\G_0$ so that $\G_0y=0$. If $\G_0$ is nonsingular, then pre-multiplying by $\G_0^{-1}$ shows that the pencil is equivalent to the eigenvalue problem $(\l_0\G_0^{-1}\G_1-\l_1I_{\tn})y=0$. Then we can take $\l_0=1$ and $\l_1$ and $y$ any eigenvalue-eigenvector pair for $\G_0^{-1}\G_1$. Thus the pencil always has a nontrivial solution (and generically $\tn$ nontrivial solutions). It follows from the proof of Theorem~\ref{thm:1-Parameter-Compressed} that $(\l_0\D_1-\l_1\D_0)z=0$ has at least one nontrivial solution with $z={\V^n}^Ty\ne0$. It then follows from the proof of Theorem~\ref{thm:1-Parameter} that the two-parameter MPP in Problem~\ref{prob:m=2} has at least one nontrivial solution (and generically $\tn$ nontrivial solutions) and that the pencils in \eqref{eqn:1-Parameter} (and therefore the pencils in \eqref{eqn:1-Parameter-Compressed}) have at least one nontrivial simultaneous solution (and generically $\tn$ nontrivial simultaneous solutions).

Suppose now that $\G$ is nonsingular.
\begin{enumerate}
\item To prove the commutativity relations, define the linear map
\begin{equation*}
\mathcal{L} := \left[\begin{array}{ccc}(A_0\!\otimes\! I_n){\V^n}^T&(A_1\!\otimes\! I_n){\V^n}^T&(A_2\!\otimes\! I_n){\V^n}^T\end{array}\right]\in\CC^{(nm)\times(3\tilde{n})},
\end{equation*}
and the subspace
\be\label{eqn:Z}
\Y:=\left\{\left[\!\!\begin{array}{c}y_0\\y_1\\y_2\end{array}\!\!\right]\!\!:y_i\in\CC^{\tilde{n}},~ \mathcal{L}\left[\!\!\begin{array}{c}y_0\\y_1\\y_2\end{array}\!\!\right]=0\right\}.
\ee 
Since the number of rows of the linear map $\mathcal{L}$ in \eqref{eqn:Z} is $nm=n(n+1)=2\tilde{n}$ (see assumptions~\eqref{eqn:Assumptions}) and since $\Y$ is the null space of $\mathcal{L}$, then
\be\label{eqn:dimZ}
\dim{\Y}\ge 3\tilde{n}-2\tilde{n}=\tilde{n}.
\ee
Let ${\K}^{nm}\!\in\!\RR^{nm\times nm}$ be the commutation matrix \cite{Magnus1979} satisfying ${\K}^{nm}(\A_i\otimes I_n)\!=\!(I_n\otimes A_i){\K}^n$, where ${\K}^n$ is defined in Lemma~\ref{lem:commutation}. Pre-multiplying the equation in \eqref{eqn:Z} by ${\K}^{nm}$ and using the fact that ${\K}^n{\V^n}^T\!\!=\!{\V^n}^T$ from the definition of $\V^n$ in \eqref{eqn:SnTn} and Property~\ref{itm:P2}  gives
\be\label{eqn:f0f1f22}
\left[\!\!\begin{array}{ccc}
(A_0\!\otimes\! I_n){\V^n}^T&\!\!(A_1\!\otimes\! I_n){\V^n}^T&\!\!(A_2\!\otimes\! I_n){\V^n}^T\\
(I_n\!\otimes\! A_0){\V^n}^T&\!\!(I_n\!\otimes\! A_1){\V^n}^T&\!\!(I_n\!\otimes\! A_2){\V^n}^T\end{array}\!\!\right]\!\!\!\left[\!\!\begin{array}{c}y_0\\y_1\\y_2\end{array}\!\!\right]\!\!=\!0,\quad\forall\! \left[\!\!\begin{array}{c}y_0\\y_1\\y_2\end{array}\!\!\right]\!\!\in\!\Y.
\ee
Pre-multiplying the first block row of \eqref{eqn:f0f1f22} by $\U^m(I_m\otimes A_2)$ and the second by $\U^m(A_2\otimes I_m)$ and subtracting gives
$$
\U^m\left((A_0\otimes A_2)\!-\!(A_2\otimes A_0)\right){\V^n}^T\!y_0\!+\!\U^m\left((A_1\otimes A_2)\!-\!(A_2\otimes A_1)\right){\V^n}^T\!y_1\!=\!0,
$$
and repeating with $A_1$ and $A_0$, we get the Kronecker determinant equations
\be\label{DeterminantEquations}
\left.\begin{array}{rcl}\G_1y_0&=&\G_0y_1,\\
\G_2y_0&=&\G_0y_2,\\
\G_2y_1&=&\G_1y_2,\end{array}\right.
\ee
from the definitions in \eqref{eqn:Dis} and \eqref{Gammas}. Thus, every vector $\left[\begin{array}{ccc}y_0^T&y_1^T&y_2^T\end{array}\right]^T$ in $\Y$ satisfies the Kronecker determinant equations in \eqref{DeterminantEquations}. It follows that
\be\label{eqn:YinYbar}
\Y\!\subseteq\!\bar{\Y}\!:=\!\left\{\left[\!\!\begin{array}{c}\bar{y}_0\\\bar{y}_1\\\bar{y}_2\end{array}\!\!\right]\!\!\!:\bar{y}_i\!\in\!\CC^{\tilde{n}},\quad
\begin{array}{c}
\G_1\bar{y}_0\!=\!\G_0\bar{y}_1\\
\G_2\bar{y}_0\!=\!\G_0\bar{y}_2\\
\G_2\bar{y}_1\!=\!\G_1\bar{y}_2
\end{array}
\right\}.
\ee

Let  $\left[\!\!\begin{array}{ccc}\bar{y}_0^T&\bar{y}_1^T&\bar{y}_2^T\end{array}\!\!\right]^T\!\!\!\in\!\bar{\Y}$. Then
$$
\G_i\bar{y}_j=\G_j\bar{y}_i,~i,j=0,1,2.
$$
Multiplying this by $\a_i$ and summing over $i$ for each $j$:
\be\label{eqn:123}
\G\bar{y}_0=\G_0\bar{y},\quad \G\bar{y}_1=\G_1\bar{y},\quad \G\bar{y}_2=\G_2\bar{y},
\ee
where $\bar{y}=\a_0\bar{y}_0+\a_1\bar{y}_1+\a_2\bar{y}_2$.
Pre-multiplying the three equations in \eqref{eqn:123} by $\G^{-1}$ gives 
\be\label{eqn:123i}
\bar{y}_0=\G^{-1}\G_0\bar{y},\quad\bar{y}_1=\G^{-1}\G_1\bar{y},\quad\bar{y}_2=\G^{-1}\G_2\bar{y}.
\ee
Thus, every vector $\left[\begin{array}{ccc}\bar{y}_0^T&\bar{y}_1^T&\bar{y}_2^T\end{array}\right]^T$ in $\bar\Y$ satisfies \eqref{eqn:123i}. It follows that
\be\label{eqn:Ybarbar}
\Y\subseteq\bar{\Y}\subseteq\bar{\bar{\Y}}:=\left\{\left[\!\!\begin{array}{c}\bar{\bar{y}}_0\\\bar{\bar{y}}_1\\\bar{\bar{y}}_2\end{array}\!\!\right]\!=\!\left[\!\!\begin{array}{c}\G^{-1}\G_0\\\G^{-1}\G_1\\\G^{-1}\G_2\end{array}\!\!\right]\!\bar{\bar{y}}
\!:\bar{\bar{y}}\!\in\!\CC^{\tilde{n}}\right\}.
\ee
Now it is clear that $\dim{\bar{\bar{\Y}}}\le\tilde{n}$ and so it follows from \eqref{eqn:dimZ} that
$$
\tilde{n}\le\dim{\Y}\le\dim{\bar{\Y}}\le\dim{\bar{\bar{\Y}}}\le\tilde{n},
$$
and so $\tilde{n}=\dim{\Y}=\dim{\bar{\Y}}=\dim{\bar{\bar{\Y}}}=\tilde{n}$. It follows that
\be\label{eqn:Y=Ybar=Ybarbar}
\Y=\bar{\Y}=\bar{\bar{\Y}}.
\ee
Next, we prove the first commutativity result: $\G^{-1}\G_1\G^{-1}\G_0=\G^{-1}\G_0\G^{-1}\G_1$. For any $y\in\CC^{\tilde{n}}$, let 
\be\label{eqn:123final}
y_0=\G^{-1}\G_0y,\quad y_1=\G^{-1}\G_1y,\quad y_2=\G^{-1}\G_2y,
\ee
so that $\left[\begin{array}{ccc}y_0^T&y_1^T&y_2^T\end{array}\right]^T\in\bar{\bar{\Y}}$.
Since $\bar{\bar{\Y}}=\bar{\Y}$, it follows from \eqref{eqn:YinYbar} that
\be\label{eqn:123finalfinal}
\G_1y_0=\G_0y_1,
\qquad\G_2y_0=\G_0y_2,\qquad\G_2y_1=\G_1y_2.
\ee
Pre-multiplying the first equation in \eqref{eqn:123final} by $\G_1$ and the second by $\G_0$ gives
$$
\G_1\G^{-1}\G_0y=\G_1y_0,\qquad\qquad
\G_0\G^{-1}\G_1y=\G_0y_1,
$$
and using the first equation in \eqref{eqn:123finalfinal} gives
$\G_1\G^{-1}\G_0y=\G_0\G^{-1}\G_1y.$
Since this is satisfied for all $y\in\CC^{\tilde{n}}$ it follows that $\G_1\G^{-1}\G_0=\G_0\G^{-1}\G_1$. Pre-multiplying by $\G^{-1}$ then proves the first commutativity result. The proof of the other commutativity results is similar (simply change the indices) and therefore omitted.
\item Note that since $\G^{-1}\G_0,\G^{-1}\G_1$ and $\G^{-1}\G_2$ commute, then \eqref{eqn:All solutions} have simultaneous solutions and can be simultaneously upper triangularized using a joint Schur decomposition; see Chapter~4 of \cite{horn1985} for more details. Next, we prove that \eqref{eqn:2-Parameter}$\Rightarrow$\eqref{eqn:All solutions}. Now, \eqref{eqn:2-Parameter} implies that
$$
\left[\!\begin{array}{ccc}A_0\!\otimes\! I_n&\!A_1\!\otimes\! I_n&\!A_2\!\otimes\! I_n\!\end{array}\right]\!\left[\!\begin{array}{c}\l_0x^{\otimes2}\\\l_1x^{\otimes2}\\\l_2x^{\otimes2}\end{array}\!\right]\!=\!0.
$$
Since $x^{\otimes2}$ is symmetric, $x^{\otimes2}={\V^n}^Ty$ for some $0\ne y\in\CC^{\tilde{n}}$. It follows that
$$
\left[\!\!\begin{array}{ccc}(A_0\!\otimes\! I_n){\V^n}^T&\!\!(A_1\!\otimes\! I_n){\V^n}^T&\!\!(A_2\!\otimes\! I_n){\V^n}^T\!\!\end{array}\right]\!\!\left[\!\!\begin{array}{c}\l_0y\\\l_1y\\\l_2y\end{array}\!\!\right]\!=\!0,
$$
and it follows from \eqref{eqn:Z} and \eqref{eqn:YinYbar} that
$$
\left[\!\!\!\begin{array}{c}y_0\\y_1\\y_2\end{array}\!\!\!\right]:=
\left[\!\!\!\begin{array}{c}\l_0y\\\l_1y\\\l_2y\end{array}\!\!\!\right]\!
\in\Y\subseteq\bar{\Y},
$$
and so it follows from \eqref{eqn:123} that
$$
\l_0\G y=\a\G_0y,\quad \l_1\G y=\a\G_1y,\quad\l_2\G y=\a\G_2y,
$$
where $\a=\a_0\l_0+\a_1\l_1+\a_2\l_2$. Since $y\ne0$, $\G$ is non-singular and not all the $\l_i$s are zero, then $\a\ne0$. Pre-multiplying by $\G^{-1}$ proves \eqref{eqn:All solutions} since we do not distinguish between eigenvalues $\l$ and $\a\l$ for $\a\ne0$.

Next, we prove \eqref{eqn:All solutions}$\Rightarrow$\eqref{eqn:2-Parameter}. Let $y\!\ne\!0$ satisfy \eqref{eqn:All solutions}. Then \eqref{eqn:Ybarbar} implies that
$$
\left[\!\!\!\begin{array}{c}y_0\\y_1\\y_2\end{array}\!\!\!\right]:=
\left[\!\!\!\begin{array}{c}\l_0y\\\l_1y\\\l_2y\end{array}\!\!\!\right]
=\left[\!\!\begin{array}{c}\G^{-1}\G_0\\\G^{-1}\G_1\\\G^{-1}\G_2\end{array}\!\!\right]\!\!y\in\bar{\bar{\Y}}.
$$
It follows from \eqref{eqn:Y=Ybar=Ybarbar} and \eqref{eqn:Z} that
$$
\left[\!\!\begin{array}{ccc}(A_0\!\otimes\! I_n){\V^n}^T&\!\!(A_1\!\otimes\! I_n){\V^n}^T&\!\!(A_2\!\otimes\! I_n){\V^n}^T\!\!\end{array}\right]\!\!\left[\!\!\begin{array}{c}y_0\\y_1\\y_2\end{array}\!\!\right]\!=\!0,
$$
which is equivalent to
$$
\left[\!\!\begin{array}{ccc}A_0\!\otimes\! I_n&\!\!A_1\!\otimes\! I_n&\!\!A_2\!\otimes\! I_n\!\!\end{array}\right]\!\!\left[\!\!\begin{array}{c}\l_0{\V^n}^Ty\\\l_1{\V^n}^Ty\\\l_2{\V^n}^Ty\end{array}\!\!\right]\!=\!0.
$$
It follows that
\be
\left((\l_0A_0+\l_1A_1+\l_2A_2)\otimes I_n\right)z=0,\qquad z:={\V^n}^Ty\ne0.
\ee
This implies that $\l_0A_0+\l_1A_1+\l_2A_2$ loses column rank and \eqref{eqn:2-Parameter} follows. If $z$ is strongly decomposable so that $z=x^{\otimes2}$, then $x$ solves \eqref{eqn:2-Parameter}. Otherwise, $x$ can be chosen as any vector in the null space of $\l_0A_0+\l_1A_1+\l_2A_2$.
\end{enumerate}
\end{proof}
\begin{remark}
As mentioned in the introduction, we use some of the machinery developed in \cite{Atkinson1972} for the solution of the MEV problem in \eqref{eqn:Aij}.
The commutation results in Part~1 of our Theorem~\ref{thm:Commutativity}  correspond to Theorem~6.7.2 in \cite{Atkinson1972} and their proof follows a similar method of proof while Part~2 corresponds to Theorems~6.6.1 and 6.8.1 in \cite{Atkinson1972}. Note also that the non-singularity of $\G$ allows us to directly go from Problem~\ref{prob:m=2} to the solution in \eqref{eqn:All solutions}.
\end{remark}
\newpage~

\section{Solution Algorithm and Examples}\label{sec:Examples} 
By way of summarizing our results, we present Algorithm~\ref{alg:SolveMEVP2} for the solution of the two-parameter MPP in Problem~\ref{prob:m=2}. We also present three examples that illustrate and clarify our solution algorithm.
%\subsection{Solution Algorithm}
\begin{algorithm}
    \caption{Computation of all solutions of the two-parameter MPP if they exist}
    \label{alg:SolveMEVP2}
    \begin{algorithmic}[1]
        \REQUIRE Matrices $A_0, A_1,A_2\!\in\!\CC^{m\times n}$
        
        \COMMENT{as described in Problem~\ref{prob:m=2}, with $m\!\ge \!n\!+\!1$.}
        \STATE Evaluate the sparse matrices $\U^m$ and $\V^n$ defined in \eqref{eqn:SnTn}.
        \STATE Evaluate the matrices $\G_0,\G_1,\G_2$ using either of the expressions in \eqref{Gammas} where $\D_0,\D_1,\D_2$ are defined in \eqref{eqn:Dis}.
        \IF[Problem~\ref{prob:m=2} has a solution.]{($m\!=\!n+1$)}
        % \STATE Problem~\ref{prob:m=2} has a solution.
        \IF{$\G\!:=\!\a_0\G_0\!+\!\a_1\G_1\!+\!\a_2\G_2$ is nonsingular for some $\a_0,\a_1,\a_2\!\in\!\CC$}
        \STATE Find all solutions $\l^{(p)}\!=\!(\l_0^{(p)},\l_1^{(p)},\l_2^{(p)})$ and $y^{(p)}$, $p=1,\ldots,k\le\tilde{n}:=\frac{n(n+1)}{2}$,
        to the simultaneous eigenvalue problems in \eqref{eqn:All solutions} using e.g. the joint upper triangular Schur form approach \cite{horn1985}.  
        \ELSE
        \STATE Find all solutions $\l^{(p)}\!=\!(\l_0^{(p)},\l_1^{(p)},\l_2^{(p)})$ and $y^{(p)}$, $p=1,\ldots,k\le\tilde{n}$,
        to the one-parameter simultaneous MPPs in \eqref{eqn:1-Parameter-Compressed} using e.g. \cite{Van1979}.
        \ENDIF
        \IF{ ${\V^n}^Ty^{(p)}$ is strongly decomposable so that ${\V^n}^Ty^{(p)}=x^{(p)}\otimes x^{(p)}$}
        \STATE $x^{(p)}$ is the eigenvector.
        \ELSE
        \STATE Choose $x^{(p)}$ to be in the null space of $\l_0^{(p)}A_0+\l_1^{(p)}A_1+\l_2^{(p)}A_2$.
        \ENDIF
        % \ENDIF
        \ELSE[($m>n+1$) and Problem~\ref{prob:m=2} has a solution if and only if \eqref{eqn:1-Parameter-Compressed} has a simultaneous solution.]
        % \COMMENT{($m>n+1$) and Problem~\ref{prob:m=2} has a solution if and only if \eqref{eqn:1-Parameter-Compressed} has a \textcolor{blue}{simultaneous} solution.}
        % \IF[Problem~\ref{prob:m=2} has a solution if and only if \eqref{eqn:1-Parameter-Compressed} has a \textcolor{blue}{simultaneous} solution.]{($m>n+1$)}
        % \STATE Problem~\ref{prob:m=2} has a solution if and only if \eqref{eqn:1-Parameter-Compressed} has a \textcolor{blue}{simultaneous} solution.
        \IF{\eqref{eqn:1-Parameter-Compressed} has a simultaneous solution}
        \STATE Find all solutions $\l^{(p)}\!=\!(\l_0^{(p)},\l_1^{(p)},\l_2^{(p)})$ and $y^{(p)}$, $p=1,\ldots,k\le\tilde{n}$,
        to the one-parameter simultaneous MPPs in \eqref{eqn:1-Parameter-Compressed} using e.g. \cite{Van1979}.
        \IF{ ${\V^n}^Ty^{(p)}$ is strongly decomposable so that ${\V^n}^Ty^{(p)}=x^{(p)}\otimes x^{(p)}$}
        \STATE $x^{(p)}$ is the eigenvector.
        \ELSE
        \STATE Choose $x^{(p)}$ to be in the null space of $\l_0^{(p)}A_0+\l_1^{(p)}A_1+\l_2^{(p)}A_2$.
        \ENDIF
        % \ELSE
        % \STATE Problem~\ref{prob:m=2} has no solution.
        \ENDIF
        \ENDIF
        \ENSURE $(\l^{(1)},x^{(1)}),\ldots,(\l^{(k)},x^{(k)})$ or no solution.
    \end{algorithmic}
\end{algorithm}
\begin{remark}\label{rem:Algorithm}
With reference to line 1, note that, while $\U^m$ and $\V^n$ contain a $\sqrt{2}$ factor (see \eqref{eqn:SnTn}), this factor can be removed using diagonal scaling and, since $\K^n$ and $\K^m$ are sparse (0,1) matrices, the matrices $\U^m$ and $\V^n$ are sparse $(0,1)$ compression matrices, which is computationally useful when the dimensions $m$ and $n$ are large. More formally, we can replace the definition of $\ \U^m$ and $\V^n$ in (and only in) \eqref{Gammas} by $-2\S^{m}_U{\F}^{m}$ and $\left[\begin{array}{c}\S^{n}_D\\2\S^{n}_L\end{array}\right]{\cH}^{n}$, respectively.
\end{remark}
\begin{remark}
With reference to line 2 regarding the numerical evaluation of Kronecker determinants in \eqref{Gammas}, note that in \cite[Chapter 4.3]{Alsubaie2019}, using the fact that $\U^m$ and $\V^n$ are effectively $(0,1)$ sparse compression matrices and using an element-by-element expression for $A_j\otimes A_k$, two algorithms were presented for evaluating the expressions depending on the available memory and computation time.
\end{remark}
\begin{remark}
With reference to line~4, we are not aware of any literature for determining whether there exists a linear combination of three general square complex matrices which is nonsingular or any efficient algorithms for determining such a combination if it exists. In Example~1 we used a simple search algorithm. In Example~2, we determined that such a combination does not exist from the special structure of the matrices. Investigating this issue in the general case is a future research direction.
\end{remark}
\begin{remark}
With reference to line 6, that is when $\G=\a_0\G_0+\a_1\G_1+\a_2\G_2$ is singular for all $\a_0,\a_1,\a_2\in\CC$, then we do not have a commutativity property and we must use Theorem~\ref{thm:1-Parameter-Compressed} for the solution. Nevertheless, it is shown in Theorem~\ref{thm:1-Parameter-Compressed} that the solutions have simultaneous eigenvectors. This suggests some partial commutativity properties. Furthermore, with reference to line 14, that is, when $m>n+1$, it is not known whether $\G$ (which is now a long matrix) having a full column rank implies any commutativity properties. Investigating these issues is a future research direction.
\end{remark}

\subsection{Example 1}
This example presents a $4\times3$ two-parameter MPP. Let
$$
A_0=\left[\begin{array}{ccc}2&3&1\\2&2&2\\4&4&3\\5&5&4\end{array}\right]\!\!,\quad
A_1=\left[\begin{array}{ccc}3&3&1\\1&2&2\\2&3&3\\4&2&4\end{array}\right]\!\!,\quad
A_2=\left[\begin{array}{ccc}3&1&2\\3&3&3\\3&4&4\\4&4&5\end{array}\right]\!\!\cdot
$$
Since $m=n+1$, then using step 3 of Algorithm~\ref{alg:SolveMEVP2}, and ignoring the $\sqrt{2}$ term in $\U^m$ and $\V^n$ following Remark~\ref{rem:Algorithm}, the Kronecker determinant matrices are:
\begin{equation*}
\G_0 \!= \!
    \left[\!\begin{array}{rrrrrr}
    12 & 14 & -2 & 22 & 8 & 12 \\
    6 & 18 & -4 & 20 & 4 & 14 \\
    0 & 20 & -6 & 28 & -2 & 22 \\
    -6 & -2 & -2 & -10 & -10 & -4 \\
    -16 & 4 & -4 & -12 & -22 & 0 \\
    -8 & 8 & -2 & -4 & -12 & 6
    \end{array}\!\right]\!\!,
\quad
\G_1 \!= \!
    \left[\!\begin{array}{rrrrrr}
    0 & -14 & 2 & -14 & 2 & -12 \\
    12 & -16 & 4 & -2 & 12 & -10 \\
    14 & -14 & 6 & 0 & 16 & -10 \\
    12 & 8 & 2 & 20 & 14 & 10 \\
    14 & 14 & 4 & 28 & 18 & 18 \\
    -2 & 8 & 2 & 6 & 0 & 8
    \end{array}\!\right]\!\!,
\end{equation*}
\begin{equation*}
\G_2 \! = \!
    \left[\!\begin{array}{rrrrrr}
    -8 & 0 & 0 & -10 & -6 & 0 \\
    -16 & -6 & 0 & -24 & -10 & -2 \\
    -14 & -18 & 0 & -28 & -10 & -6 \\
     0 & -4 & 0 & -4 & -2 & -4 \\
    6 & -12 & 0 & -6 & 4 & -12 \\
    12 & -14 & 0 & -2 & 10 & -10
    \end{array}\!\right]\!\!\cdot
\end{equation*}
% Note that the elements of the $\G_i$'s are rearrangements of pair-wise products of elements of $A_j$ and $A_k$; see point 2 of Remark~\ref{rem:Algorithm}. 
Since $\G_0$ is nonsingular, then we can use Theorem~\ref{thm:Commutativity} (see step 4 of Algorithm~\ref{alg:SolveMEVP2}) to get all 6 ($=\tn$) eigenvalues $\l$ as
\begin{equation*}
\left\{\!\!
\left[\!\!\!\begin{array}{r}1\\-1\\0\end{array}\!\!\!\right]\!\!\!, \!
\left[\!\!\!\begin{array}{c}1\\-0.0674\\\phantom{-}0.2755\end{array}\!\!\!\right]\!\!\!,
\!
\left[\!\!\!\begin{array}{c}1\\\phantom{-}1.1714\\-1.5777\end{array}\!\!\!\right]\!\!\!,
\!
\left[\!\!\!\begin{array}{c}1\\-0.8627 - 0.1011i\\\phantom{-}1.2215 - 0.8717i\end{array}\!\!\!\right]\!\!\!,
\!
\left[\!\!\!\begin{array}{c}1\\-0.8627 \!+\! 0.1011i\\\phantom{-}1.2215 \!+\! 0.8717i\end{array}\!\!\!\right]\!\!\!,
\!
\left[\!\!\!\begin{array}{c}1\\-1.1025\\\phantom{-}0.0630\end{array}\!\!\!\right]\!\!
\right\}\!\cdot
\end{equation*}

\newpage
\subsection{Example 2}
This example presents a $3\times2$ two-parameter pencil with a continuum of solutions and demonstrates that $m \geq n + 1$ is only a necessary condition for a zero-dimensional solution set. Let
$$
A_0=\left[\begin{array}{cc}
1&0\\4&0\\7&0\end{array}\right]\!\!,\quad
A_1=\left[\begin{array}{cc}
2&0\\5&0\\8&0\end{array}\right]\!\!,\quad
A_2=\left[\begin{array}{cc}
3&0\\0&6\\9&0\end{array}\right]\!\!\cdot
$$
Note that $A_0$ and $A_1$ have zero second columns, so there is a continuum of solutions $\l_0$ and $\l_1$ arbitrary, $\l_2=0$ and $x=\left[\begin{array}{cc}0&1\end{array}\right]^T
$. In this example $m=n+1$, then using step 3 of Algorithm~\ref{alg:SolveMEVP2}, the Kronecker determinants are
\begin{equation*}
\G_0 = 
    \left[\begin{array}{rrr}
    -30&0&24\\
    -12&0&0\\
    90&0&-96
    \end{array}\right]\!\!,
\quad
\G_1 = 
    \left[\begin{array}{rrr}
    24&0&-12\\
    24&0&0\\
    -72&0&84
    \end{array}\right]\!\!,
\quad
\G_2 = 
    \left[\begin{array}{rrr}
    -6&0&0\\
    -12&0&0\\
    -6&0&0
    \end{array}\right]\!\!\cdot
\end{equation*}
Note that the three Kronecker determinants have a zero second column and so $\a_0\G_0+\a_1\G_1+\a_2\G_2$ is singular for all $\a_0,\a_1,\a_2\in\CC$ and we need to use Theorem~\ref{thm:1-Parameter-Compressed}, which follows from Theorem~\ref{thm:1-Parameter} (see step 6 of Algorithm~\ref{alg:SolveMEVP2}).
Solving $(\l_i\G_j - \l_j\G_i)y=0$, $(i,j)\in\{(0,1),(0,2),(1,2)\}$ by reducing the pairs $\{\G_i,\G_j\}$ to the Kronecker Canonical Form \cite{Van1979} or otherwise, we get the two simultaneous solutions:
\bean
\left.\begin{array}{llll}\l^{(1)}\!=\!\left[\!\begin{array}{c}\star\\\star\\0\end{array}\!\right]\!, &y^{(1)} \!=\! \matThrOne{\!0\!}{\!1\!}{\!0\!}\!,&z^{(1)} \!=\! \! \matFourOne{\!0\!}{\!0\!}{\!0\!}{\!1\!} \!=\! \matTwoOne{\!0\!}{\!1\!} \!\otimes\! \matTwoOne{\!0\!}{\!1\!}&\Rightarrow x^{(1)}\!=\!\left[\!\begin{array}{c}0\\1\end{array}\!\right]\!,\\ \noalign{\vskip2pt}
\l^{(2)}\!=\!\left[\!\!\begin{array}{c}1\\-2\\1\end{array}\!\!\right]\!,&y^{(2)} \!=\! \matThrOne{\!\!1\!\!}{\!\!a\!\!}{\!\!1\!\!}\!,&z^{(2)} \!=\!\! \matFourOne{\!1\!}{\!1\!}{\!1\!}{\!a\!}\! =\! \matTwoOne{\!\!1\!\!}{\!\!1\!\!}\! \otimes \!\matTwoOne{\!\!1\!\!}{\!\!1\!\!}&\Rightarrow x^{(2)}\!=\!\left[\!\!\begin{array}{c}1\\1\end{array}\!\!\right]\!,
\end{array}\right.
\eean
where $\star$ denotes an arbitrary complex number and where we have normalised the second eigenvalue such that $\l_0=1$. The second eigenvector $y^{(2)}$ has an arbitrary entry, denoted as $a$. Theorem~\ref{thm:1-Parameter} (via Theorem~\ref{thm:m=2-square}) guarantees that there exists a strongly decomposable $z^{(2)}={\V^n}^Ty^{(2)}$; choosing $a=1$ ensures $z^{(2)}$ is strongly decomposable. Note that we get all the solutions using Theorem~\ref{thm:1-Parameter-Compressed}, Theorem~\ref{thm:1-Parameter} and Theorem~\ref{thm:m=2-square} even though $\a_0\G_0+\a_1\G_1+\a_2\G_2$ is singular for all $\a_0,\a_1$ and $\a_2$.

\subsection{Example 3}
This example presents a $4\times2$ two-parameter pencil. Let
$$
A_0=\left[\begin{array}{cc}
2&4\\6&0\\0&2\\6&0\end{array}\right],\quad
A_1=\left[\begin{array}{cc}
1&0\\0&1\\0&0\\0&0\end{array}\right],\quad
A_2=\left[\begin{array}{cc}
0&0\\2&0\\0&2\\2&0\end{array}\right]\!\cdot
$$
% Note that $A_0$ and $A_1$ have zero second columns, so there is a continuum of solutions $\l_0$ and $\l_1$ arbitrary, $\l_2=0$ and $x=\left[\begin{array}{cc}0&1\end{array}\right]^T$.
In this example $m>n+1$, then using our approach in Theorem~\ref{thm:1-Parameter-Compressed}, which follows from Theorem~\ref{thm:1-Parameter} (see step 14 of Algorithm~\ref{alg:SolveMEVP2}):
\begin{equation*}
\G_0 = 
    \left[\begin{array}{rrr}
    4&0&0\\
    0&0&4\\
    4&0&0\\
    0&4&0\\
    0&0&4\\
    0&0&0
    \end{array}\right]\!\!,
\quad
\G_1 = 
    \left[\begin{array}{rrr}
    -8&0&-16\\
    0&-16&-8\\
    -8&0&-16\\
    0&0&-16\\
    0&0&0\\
    0&0&16
    \end{array}\right]\!\!,
\quad
\G_2 = 
    \left[\begin{array}{rrr}
    -12&8&4\\
    0&0&-4\\
    -12&0&0\\
    0&-4&0\\
    0&0&-12\\
    0&0&0
    \end{array}\right]\!\!\cdot
\end{equation*}
Solving $(\l_i\G_j - \l_j\G_i)y=0$, $(i,j)\in\{(0,1),(0,2),(1,2)\}$ by reducing the pairs $\{\G_i,\G_j\}$ to the Kronecker Canonical Form, we get the single simultaneous solution:
\begin{equation*}
    \l = \matThrOne{\phantom{-}1}{-2}{-3}\!\!, \quad y = \matThrOne{1}{0}{0}\!\!,\quad z=\left[\begin{array}{c}1\\0\\0\\0\end{array}\right]=\matTwoOne{1}{0} \otimes \matTwoOne{1}{0}\!\!,\quad\Rightarrow x=\left[\begin{array}{c}1\\0\end{array}\right]\!\!\cdot
\end{equation*}

\section{The Multiparameter MPP}\label{sec:Future} In the previous sections, we presented the solution of the two-parameter MPP. By way of outlining our future research directions, we summarize in this section the known results about  
the $r$-parameter MPP for $r\!>\!2$: given $(r\!+\!1)$ matrices $A_0,\ldots,A_r\!\in\!\CC^{m\times n}$, with $m\!\ge\! n\!+\!r\!-\!1$, assume that
$$
\rrank\left(\left[\begin{array}{c}A_0\\\vdots\\A_r\end{array}\right]\right)={n},\qquad\rrank\left(\left[\begin{array}{ccc}A_0&\cdots&A_r\end{array}\right]\right)={m}.
$$
Find all eigenvalues $\l$ and the corresponding eigenvectors $x$ such that \eqref{eqn:Pencil} is satisfied.
Our preliminary work in \cite{Alsubaie2019} gives the following result:
\begin{theorem}
For $i=0,\ldots,r$, define the Kronecker commutator operators
$$
\D_i=\left|\left[\begin{array}{ccccccc}A_0&\cdots&A_{i-1}&\hat{A}_i&A_{i+1}&\cdots&A_r\\
\vdots&\vdots&\vdots&\vdots&\vdots&\vdots&\vdots\\
A_0&\cdots&A_{i-1}&\hat{A}_i&A_{i+1}&\cdots&A_r
\end{array}\right]\right|_{\otimes}\in\CC^{m^r\times n^r},
$$
where $\hat{A}_i$ indicates column deletion and where the matrix defining $\D_i$ has $r$ identical block rows. Then
\begin{enumerate}
\item[(a)] $
\left(\sum\limits_{i=1}^r\l_i A_i\right)\!x\!=\!0,~~0\!\ne\!\left[\!\!\begin{array}{c}\l_1\\\vdots\\\l_r\end{array}\!\!\right]\!\!\in\!\CC^{r},~~0\!\ne\! x\!\in\!\CC^n\Longleftrightarrow \D_0 x^{\otimes r}=0,~~0\ne x\in\CC^{n}$.
\item[(b)] \eqref{eqn:Pencil}$~\Longleftrightarrow (\l_0\D_j-\l_j\D_0)x^{\otimes r}=0,~~0\ne x\in\CC^{n},~~j=1,\ldots,r
$.
\item[(c)] Suppose we index the rows and columns of $\D_i$ as $(i_1,\ldots,i_r)$ and $(j_1,\ldots,j_r)$ with $1\le i_1,\ldots,i_r\le m$ and $1\le j_1,\ldots,j_r\le n$, respectively, and define the Kronecker determinants $\G_i\in\CC^{\binom{m}{r}\times\displaymultiset{n}{r}}$, obtained as compressed versions of $\D_i$ by retaining the $\binom{m}{r}$ rows of $\D_i$ with indices $(i_1,\ldots,i_r)$ with $1\le i_1<\ldots< i_r\le m$ and the $\displaymultiset{n}{r}$ columns of $\D_i$ with indices $(j_1,\ldots,j_r)$ with $1\le j_1\le\ldots\le j_r\le n$. Then 
\be\label{eqn:r-Parameter-Small}
\eqref{eqn:Pencil}\Longleftrightarrow (\l_0\G_j-\l_j\G_0)y=0,~0\ne y\in\CC^{\displaymultiset{n}{r}},~j=1,\ldots,r,
\ee
with $y$ a compressed version (in the above sense) of $x^{\otimes r}$ for some $\x\in\CC^n$ and where $
\binom{n}{k}\!=\!\frac{n!}{k!(n-k)!}$ and 
$\multiset{n}{k}\!=\!\binom{n+k-1}{k}=\frac{(n+k-1)!}{k!(n-1)!}$
denote the number of $k$-combinations of the set $\{1,\ldots,n\}$ without and with replacement, respectively.
\item[(d)] (See also \cite{shapiro2009}) If $m=n+r-1$, then $\binom{m}{r}=\displaymultiset{n}{r}$ and the one-parameter MPPs in \eqref{eqn:r-Parameter-Small} have at least one solution (and generically $\binom{m}{r}=\displaymultiset{n}{r}$ solutions). 
\end{enumerate}
\end{theorem}
The above result provides a solution for the $r$-parameter MPP in the generic case, which here can be defined as $m=n+r-1$ and $\G:=\sum_{i=0}^r\a_i\G_i$ nonsingular for some $\a_0,\ldots,\a_r\in\CC$. However, a complete solution in the case that $m>n+r-1$ or $m=n+r-1$ but $\G:=\sum_{i=0}^r\a_i\G_i$ is singular for all $\a_0,\ldots,\a_r\in\CC$ requires further research for the proof of the above results with the strongly decomposable vector $x^{\otimes r}$ replaced by a general $z\in\CC^{n^r}$ and the compressed vector $y$ replaced by a general vector. This would allow us to extend Algorithm~\ref{alg:SolveMEVP2} to the case $r>2$.

\section{Conclusions}\label{sec:Conclusions}

In this paper, we have considered the multiparameter matrix pencil problem (MPP) and highlighted its relation to earlier problems. Interest in the problem is relatively new and to the authors' knowledge little work has been done for its solution. We were led to it because of its links to the optimal $\mathcal{H}_2$ model reduction problem \cite{Serkan2008,beattie2017,Ahmad2010,Ahmad2011},  established in the general case in \cite{Alsubaie2019}. Due to its inherent importance and as a first step towards extending the solution approach to the general multiparameter MPP, we presented a full solution of the two-parameter MPP. Firstly, an inflation process was implemented in Theorem~\ref{thm:1-Parameter} to prove (using Theorem~\ref{thm:m=2-square}) that the two-parameter MPP is equivalent to three $m^2\times n^2$ simultaneous one-parameter MPPs. The inflated MPPs are given in terms of Kronecker commutator operators involving the original matrices which exhibit several symmetries. These symmetries were thoroughly investigated in Section~\ref{sec:Commutator} as they can also be interesting in other problems than the MPP. The results were then exploited in Theorem~\ref{thm:1-Parameter-Compressed} to deflate the dimensions of the one-parameter MPPs to $\frac{m(m-1)}{2}\times\frac{n(n+1)}{2}$ using Theorem~\ref{thm:Diagonalisation}, thus simplifying their numerical solution. In the case that $m=n+1$ (that is, when the number of equations is equal to the number of unknowns), Theorem~\ref{thm:Commutativity} showed that a solution to the two-parameter MPP always exists and established a commutativity property of the Kronecker determinants under a rank assumption. This was used to decouple the three one-parameter MPPs into three simultaneous eigenvalue problems thus simplifying their solution, which led to Algorithm~\ref{alg:SolveMEVP2}. Numerical examples were then given to highlight the procedure of the proposed solution algorithm. Finally, future research directions to extend our approach to the general multiparameter MPP were outlined.
\section*{Acknowledgement} We thank the reviewers for their substantial time and effort which greatly helped us to improve the readability and quality of the paper.
\bibliographystyle{siamplain}
\bibliography{references}

\end{document}